\documentclass[11pt,reqno]{amsart}
\usepackage{amsfonts}
\usepackage{amssymb}
\usepackage{amscd}
\usepackage{hyperref}
\usepackage{color}

\usepackage{xcolor,cancel}
 
\def \bui#1#2{\mathrel{\mathop{\kern 0pt#1}\limits^{#2}}}

\newcommand{\func}[1]{\operatorname{#1}}
\newcommand{\limfunc}[1]{\operatorname{#1}}
\newtheorem{theorem}{Theorem}[section]
\newtheorem{corollary}[theorem]{Corollary}
\newtheorem{lemma}[theorem]{Lemma}
\newtheorem{proposition}[theorem]{Proposition}

\newtheorem{definition}[theorem]{Definition}
\newtheorem{remark}[theorem]{Remark}
\newtheorem{notation}[theorem]{Notation}

\newtheorem{example}[theorem]{Example}
\newtheorem{problem}{Problem}
\numberwithin{equation}{section}

\begin{document}
\title{New cohomological invariants of foliations}
\author[G.~Habib]{Georges Habib}
\address{Lebanese University \\
Faculty of Sciences II \\
Department of Mathematics\\
P.O. Box 90656 Fanar-Matn \\
Lebanon and Universit\'e de Lorraine, CNRS, IECL, F-54506, Nancy, France}
\email[G.~Habib]{ghabib@ul.edu.lb}
\author[K.~Richardson]{Ken Richardson}
\address{Department of Mathematics \\
Texas Christian University \\
Fort Worth, Texas 76129, USA}
\email[K.~Richardson]{k.richardson@tcu.edu}
\subjclass[2010]{57R30; 53C12; 58A14}
\keywords{foliation, cohomology, homotopy invariance, Hodge theory}
\thanks{This work was supported by a grant from the Simons Foundation (Grant
Number 245818 to Ken Richardson)}

\begin{abstract}
Given a smooth foliation on a closed manifold, basic forms are differential
forms that can be expressed locally in terms of the transverse variables.
The space of basic forms yields a differential complex, because the exterior
derivative fixes this set. The basic cohomology is the cohomology of this
complex, and this has been studied extensively. Given a Riemannian metric,
the adjoint of the exterior derivative maps the orthogonal complement of the
basic forms to itself, and we call the resulting cohomology the ``antibasic
cohomology''. Although these groups are defined using the metric, the
dimensions of the antibasic cohomology groups are invariant under
diffeomorphism and metric changes. If the underlying foliation is
Riemannian, the groups are foliated homotopy invariants that are independent
of basic cohomology and ordinary cohomology of the manifold. For this class
of foliations we use the codifferential on antibasic forms to obtain the
corresponding Laplace operator, develop its analytic properties, and prove a
Hodge theorem. We then find some topological and geometric properties that
impose restrictions on the antibasic Betti numbers.
\end{abstract}

\maketitle

\section{Introduction}

The ordinary Hodge decomposition theorem on a closed Riemannian manifold $%
\left( M,g\right) $ of dimension $n$ gives an $L^{2}$-orthogonal
decomposition of differential forms:%
\begin{equation*}
\Omega ^{k}\left( M\right) =\func{im}\left( d_{k-1}\right) \oplus \mathcal{H}%
^{k}\oplus \func{im}\left( \delta _{k+1}\right) ,~0\leq k\leq n,
\end{equation*}%
where $d_{k}:\Omega ^{k}\left( M\right) \rightarrow \Omega ^{k+1}\left(
M\right) $ is the exterior derivative with $L^{2}$-adjoint $\delta
_{k+1}:\Omega ^{k+1}\left( M\right) \rightarrow \Omega ^{k}\left( M\right) $%
, and where $\mathcal{H}^{k}=\ker \left( \Delta _{k}\right) $ is the space
of harmonic $k$-forms. Note also that%
\begin{equation*}
\ker \left( d_{k}\right) =\func{im}\left( d_{k-1}\right) \oplus \mathcal{H}%
^{k}\text{ and }\ker \left( \delta _{k}\right) =\mathcal{H}^{k}\oplus \func{%
im}\left( \delta _{k+1}\right) .
\end{equation*}%
From this we get that the de Rham cohomology groups satisfy $H^{k}\left(
M\right) \cong \mathcal{H}^{k}$. Now, an alternative way of looking at this
is to define a \textquotedblleft new\textquotedblright\ de Rham homology $%
H_{\delta }^{k}\left( M\right) $ using $\delta $ instead of $d$: $\delta
^{2}=0$, so 
\begin{equation*}
H_{\delta }^{k}\left( M\right) =\frac{\ker \delta _{k}}{\func{im}\delta
_{k+1}},~0\leq k\leq n,
\end{equation*}%
is well-defined. By the equations above for $\ker $ $\delta _{k}$ and $\ker
d_{k}$, $H_{\delta }^{k}\left( M\right) \cong H_{d}^{k}\left( M\right) $. So
no one ever defines $H_{\delta }^{k}\left( M\right) $ separately, because it
does not provide anything new, and it seems to require a metric.

We consider however the situation where $M$ is endowed with a smooth
foliation $\mathcal{F}$ of codimension $q$. Many researchers have studied
the properties of basic forms on foliations (see \cite{Re0} for the original
work and the expositions \cite{Re1}, \cite{Mo}, \cite{Tond} and the
references therein). Specifically, the basic forms are differential forms on 
$M$ that locally depend only on the transverse variables. Because the
exterior derivative preserves the set $\Omega _{b}^{\ast }\left( M\right) $
of basic forms, one can define the basic cohomology groups as 
\begin{equation*}
H_{b}^{k}\left( M,\mathcal{F}\right) :=\frac{\ker \left( d:\Omega
_{b}^{k}\left( M\right) \rightarrow \Omega _{b}^{k+1}\left( M\right) \right) 
}{\func{im}\left( d:\Omega _{b}^{k-1}\left( M\right) \rightarrow \Omega
_{b}^{k}\left( M\right) \right) },~0\leq k\leq q.
\end{equation*}%
These cohomology groups are invariants of the foliation and can in general
be infinite dimensional even when $M$ is compact. The isomorphism classes of
these groups are invariant under any homotopy equivalence between foliations
that preserve the leaves. For certain classes of foliations, such as
Riemannian foliations, these cohomology groups are finite dimensional.

Let $L^{2}\left( \Omega _{b}^{\ast }\left(
M\right) \right) $ denote the completion of the space of smooth basic forms 
with respect to the $L^2$ inner product on differential forms on $M$. 
This is a subspace of the Hilbert space of differential forms with respect to 
this same inner product.
Since the latter Hilbert space is complete, the subspace is the same as the 
closure of the space  of smooth basic forms with respect to the $L^2$ norm.
Since $d$ preserves the smooth basic forms as mentioned previously, 
the formal adjoint $\delta $ of $d$
with respect to the $L^2$ inner product preserves the smooth forms inside 
the $L^2$ orthogonal 
complement $\left( \Omega
_{b}^{\ast }\left( M\right) \right) ^{\bot }$, and we denote the set of
smooth forms in this subspace by $\Omega _{a}^{\ast }\left( M,g\right) $,
the set of \textquotedblleft antibasic forms\textquotedblright . Because $%
\delta ^{2}=0$ on this space, we may define the \textquotedblleft antibasic
cohomology groups\textquotedblright\ as%
\begin{equation*}
H_{a}^{k}\left( M,\mathcal{F},g\right) :=\frac{\ker \left( \delta :\Omega
_{a}^{k}\left( M,g\right) \rightarrow \Omega _{a}^{k-1}\left( M,g\right)
\right) }{\func{im}\left( \delta :\Omega _{a}^{k+1}\left( M,g\right)
\rightarrow \Omega _{a}^{k}\left( M,g\right) \right) },~0\leq k\leq n.
\end{equation*}%
We see that $H_{a}^{k}\left( M,\mathcal{F},g\right) $ depends on the choice
of $g$, but we show that the isomorphism classes of these groups are
independent of this choice (Theorem \ref{indepOfMetricTheorem}) and are in
fact invariants of the smooth foliation structure (Corollary \ref%
{diffeoInvtCorollary}). For that reason, we henceforth remove the background
metric $g$ from the notation. Unlike the case of the de Rham cohomology of
ordinary manifolds defined using $\delta $ above, these cohomology groups
provide new invariants of the foliation, which are not necessarily isomorphic to
either the basic or ordinary de Rham cohomology groups.

We are interested in whether these new foliation invariants give
obstructions to certain types of geometric structures on the manifolds and
foliations. In Theorem \ref{codim1Theorem} we show that if the foliation is
codimension one on a connected manifold, and if the mean curvature form of
the normal bundle is everywhere nonzero, then $H_{a}^{0}\left( M\right)
=\left\{ 0\right\} $, and $H_{a}^{j}\left( M\right) =H^{j}\left( M\right) $
for $j\geq 1$.

Starting with Section \ref{RiemannianFoliationSection}, we consider the case
of Riemannian foliations, where the normal bundle carries a holonomy
invariant metric; c.f. \cite{Rein}, \cite{Mo}, \cite{Tond}. As is customary,
we choose a bundle-like metric, one such that the leaves of the foliation
are locally equidistant. In this particular case, the geometry forces many
consequences for the antibasic cohomology. One crucial property of
Riemannian foliations that allows us to proceed with analysis is that the $%
L^{2}$ orthogonal projection $P_{b}$ from all forms to basic forms preserves
smoothness. This was shown in \cite{PaRi} and \cite{AL}, and it is false in general for
non-Riemannian foliations (see Example \ref{ProjectionNotSmoothExample}). As
a consequence, it is also true that the $L^{2}$ orthogonal projection $P_{a}$
from all forms to antibasic forms preserves smoothness. In Proposition \ref%
{commutatorProposition}, we derive explicit formulas for the commutators $%
[d,P_{a}]$ and $[\delta ,P_{a}]$, which are zeroth order operators that are
in general not pseudodifferential. These formulas allow us to express the
antibasic Laplacian $\Delta _{a}=\left( P_{a}\left( d+\delta \right)
P_{a}\right) ^{2}$ in terms of elliptic operators on all forms in Theorem %
\ref{antibasicLaplacianTheorem}. That is, $\Delta _{a}$ can be written in
terms of the ordinary Laplacian $\Delta $ on $M$ by the formula%
\begin{equation*}
\Delta _{a}=\left( \Delta +\delta P_{b}\varepsilon ^{\ast }+P_{b}\varepsilon
^{\ast }\delta \right) P_{a},
\end{equation*}%
where $\varepsilon ^{\ast }$ is a zeroth order differential operator
determined by the geometry of the foliation and defined explicitly in
Proposition \ref{commutatorProposition}.

Because $\Delta _{a}$ and $D_{a}=P_{a}\left( d+\delta \right) P_{a}$ are
similar to elliptic differential operators but are in general not
pseudodifferential, we do not necessarily expect them to satisfy the usual
properties of Laplace and Dirac operators. However, in Section \ref%
{FunctionalAnalysisSection}, we are able to show many of the functional
analysis results with a few modifications. Specifically, we prove a version
of G\aa rding's Inequality (Lemma \ref{GardingIneqLemma}), the elliptic
estimates (Lemma \ref{ellipticEstimatesLemma}), and the essential
self-adjointness of both $D_{a}$ and $\Delta _{a}$ (Corollary \ref{Essential
Self-Adjointness Corollary}). Also, we show that elliptic regularity holds
(Proposition \ref{Elliptic Regularity Proposition}), and finally we prove
the spectral theorem (Theorem \ref{spectralTheorem}) for $\Delta
_{a}=D_{a}^{2}$ and $D_{a}$, showing that there exists a complete
orthonormal basis of $L^{2}\left( \Omega _{a}^{\ast }\left( M\right) \right) 
$ consisting of smooth eigenforms of $D_{a}$, and the eigenvalues of $\Delta
_{a}$ have finite multiplicity and accumulate only at $+\infty $. In all of
these cases, the proofs are a bit more complicated than usual because of the
antibasic projection and the issue of operators not being pseudodifferential.

In Section \ref{antibasicHodgeSection}, we are able to prove the Hodge
theorem and decomposition (Theorem \ref{AnitbasicHodgeTheorem} and Corollary %
\ref{HodgeOrthogDecompCorollary}) for the antibasic forms, again only for
the Riemannian foliation case. For these foliations, there is an alternate
way of expressing the antibasic cohomology, using $d_{a}=P_{a}dP_{a}$ as a
differential. Then it turns out that if $f:(M,\mathcal{F})\rightarrow
(M^{\prime },\mathcal{F}^{\prime })$ is a foliated map, which takes leaves
into leaves, then $P_{a}f^{\ast }P_{a}^{\prime }$ induces a linear map on $%
d_{a}$-cohomology. We show that for Riemannian foliations, the
isomorphism classes of antibasic cohomology groups are foliated homotopy
invariants; see Theorem \ref{homotopyAxiomTheorem} and Corollary \ref%
{homotopyInvarianceCorollary}. We do know in general that the antibasic
Betti numbers are foliated diffeomorphism invariants, but it is an open
question whether they are foliated homotopy invariants; see Problem \ref%
{OpenHomotopyQuestion} and the preceding discussion.

In Section \ref{propertiesApplicationsSection}, we prove identities for
antibasic cohomology in special cases. If the foliation is Riemannian, then $%
H_{a}^{0}\left( M,\mathcal{F}\right)\cong \left\{ 0 \right\}$ and%
\begin{equation*}
\dim H^{1}\left( M\right) \leq \dim H_{b}^{1}\left( M,\mathcal{F}\right)
+\dim H_{a}^{1}\left( M,\mathcal{F}\right) ;
\end{equation*}%
see Proposition \ref{H0_Prop} and Proposition \ref{degree 1 cohomology Prop}%
. If in addition the normal bundle is involutive, then for all $k$, 
\begin{equation*}
H^{k}\left( M\right) \cong H_{b}^{k}\left( M,\mathcal{F}\right) \oplus
H_{a}^{k}\left( M,\mathcal{F}\right) ,
\end{equation*}%
by Proposition \ref{MeanBasicInvNormalBdleProp}. In the special case where
the Riemannian foliation is the set of orbits of a connected, compact Lie
group of isometries, we show that antibasic cohomology can be computed using
only the subspace of invariant differential forms; see Proposition \ref%
{groupActionProp}.

The case of Riemannian flows is investigated in Section \ref%
{RiemannianFlowSection}. In this setting, we are able to characterize $%
H_{a}^{1}\left( M,\mathcal{F}\right) $. We prove in Proposition \ref%
{Prop_r2r+1} that when the flow is taut, meaning that there exists a metric
for which the leaves are minimal, 
\begin{equation*}
\dim \left( H_{a}^{r+1}\left( M,\mathcal{F}\right) \right) \geq \dim \left(
H_{b}^{r}\left( M,\mathcal{F}\right) \right)
\end{equation*}%
for $r\geq 0$. In the particular case where $H^{1}\left( M\right) =\left\{
0\right\} $ and $M$ is connected, we get $H_{a}^{1}\left( M\right) \cong 
\mathbb{R}$ always (Theorem \ref{H1zeroKbasicProp}). On the other hand, if $%
M $ is connected and the flow is nontaut, we have that $H_{a}^{1}\left( M\right) \cong \{0\}$.

In Section \ref{ExamplesSection}, we compute the antibasic cohomology of
specific foliations in low dimensions. These examples include Riemannian and
non-Riemannian foliations and illustrate the results we have proved. 

\section{Basic and antibasic cohomology of foliations\label%
{AntiBasicFirstSection}}

Let $M$ be a smooth, closed manifold, and let $\mathcal{F} $ be a smooth foliation on $M$ of codimension $q$
and dimension $p$ (i.e. the dimension of $M$ is $n=p+q$). 
The subspace $\Omega _{b}\left( M\right) \subseteq \Omega
\left( M\right) $ of basic differential forms is defined as 
\begin{equation*}
\Omega _{b}\left( M\right) =\left\{ \beta \in \Omega \left( M\right)
:X\lrcorner \beta =0,X\lrcorner d\beta =0\text{ for all }X\in \Gamma \left( T%
\mathcal{F}\right) \right\}
\end{equation*}%
where $X\lrcorner $ denotes interior product with $X$. Since $d$ maps basic
forms to themselves, we may compute the basic cohomology 
\begin{equation*}
H_{b}^{k}\left( M,\mathcal{F}\right) :=\frac{\ker \left( d:\Omega
_{b}^{k}\left( M\right) \rightarrow \Omega _{b}^{k+1}\left( M\right) \right) 
}{\func{im}\left( d:\Omega _{b}^{k-1}\left( M\right) \rightarrow \Omega
_{b}^{k}\left( M\right) \right) }
\end{equation*}%
for $0\leq k\leq q$. Because 
$d$ commutes with pullbacks, these vector spaces are smooth invariants of the
foliation, meaning that foliated diffeomorphisms (diffeomorphisms that map
leaves onto leaves) preserve the basic cohomology groups.
In \cite{EKN}, it was shown for complete Riemannian foliations that 
the basic cohomology algebra is a topological invariant. In general, it is
not true that the basic cohomology is topologically invariant, because there
exist smooth foliations that are foliated homeomorphic (but not foliated 
diffeomorphic) with nonisomorphic basic cohomology groups (see the
introduction of \cite{EKN} for an example).
However, in \cite[Th\'eor\`eme 1, Corollaire 1]{BenRA}, the authors show 
for arbitrary smooth foliations that smooth foliated homotopy equivalences induce 
isomorphisms on basic cohomology.

In general, $H_{b}^{k}\left( M,\mathcal{F}\right) $ need not be
finite dimensional, unless there are topological restrictions (such as the
existence of a bundle-like metric), and even when such restrictions apply,
Poincar\'{e} duality is not satisfied except in special cases, such as when
the foliation is taut and Riemannian. Much work on these cohomology groups
has been done (c.f. \cite{AL}, \cite{Mo}, \cite{Tond}, \cite{BenRA}, \cite%
{HabRic2}, \cite{HabRic8000}, and the associated references).

Suppose next that $M$ is endowed with a Riemannian metric $g$. For simplicity, we will assume that $M$ is oriented and $\mathcal{F}$ is transversally oriented
in what follows, to make the Hodge star operators well-defined. However, many of the 
results carry over to the general case with minor changes and simple adjustments to the
proofs.

A metric on
the bundle of differential forms is induced from $g$, and in fact for any $%
\alpha ,\beta \in \Omega ^{r}\left( M\right) $,%
\begin{equation*}
\left\langle \alpha ,\beta \right\rangle =\int_{M}\alpha \wedge \ast \beta ,
\end{equation*}%
where $\ast $ is the Hodge star operator. The formal adjoint $\delta $ of $d$
with respect to this metric satisfies%
\begin{equation*}
\delta =\left( -1\right) ^{nr+n+1}\ast d\ast =\left( -1\right) ^{r}\ast
^{-1}d\ast
\end{equation*}%
on $\Omega ^{r}\left( M\right) $. We let the smooth part of the $L^{2}$%
-orthogonal complement of $\Omega _{b}^{r}\left( M\right) $ be 
\begin{equation*}
\Omega _{a}^{r}\left( M,g\right) :=\Omega _{b}^{r}\left( M\right) ^{\bot
}=\left\{ \alpha \in \Omega ^{r}\left( M\right) :\left\langle \alpha ,\beta
\right\rangle =0\text{ for all }\beta \in \Omega _{b}^{r}\left( M\right)
\right\}
\end{equation*}%
the space of \textbf{antibasic }$r$\textbf{-forms}. Observe that for all $%
\beta \in \Omega _{b}^{r-1}\left( M\right) $, $\alpha \in \Omega
_{a}^{r}\left( M,g\right) $,%
\begin{equation*}
0=\left\langle d\beta ,\alpha \right\rangle =\left\langle \beta ,\delta
\alpha \right\rangle ,
\end{equation*}%
so that $\delta $ preserves the antibasic forms, and again $\delta ^{2}=0$.
We now define the \textbf{antibasic cohomology groups} $H_{a}^{r}\left( M,%
\mathcal{F},g\right) $ for $0\leq r\leq n$ by%
\begin{equation*}
H_{a}^{r}\left( M,\mathcal{F},g\right) :=\frac{\ker \left( \delta :\Omega
_{a}^{r}\left( M,g\right) \rightarrow \Omega _{a}^{r-1}\left( M,g\right)
\right) }{\func{im}\left( \delta :\Omega _{a}^{r+1}\left( M,g\right)
\rightarrow \Omega _{a}^{r}\left( M,g\right) \right) }.
\end{equation*}

\begin{theorem}
\label{indepOfMetricTheorem}Let $\mathcal{F}$ be a smooth foliation on a
smooth, closed, oriented manifold $M$ that is endowed with a metric $%
g$. The isomorphism classes of the groups $H_{a}^{r}\left( M,\mathcal{F}%
,g\right) $ do not depend on the choice of $g$ and are thus invariants of $%
\left( M,\mathcal{F}\right) $.

\begin{proof}
Consider a general change of metric from $g$ to $g^{\prime }$. Let $\ast $
denote the Hodge star operator for metric $g$, and let $\ast ^{\prime }$
denote the Hodge star operator for $g^{\prime }$. Similarly we define $%
\delta $ and $\delta ^{\prime }$. We define the invertible bundle maps $%
A^{r} $ and $B^{r}$ on $\Omega ^{r}\left( M\right) $ by%
\begin{eqnarray*}
A^{r} &:&=\ast ^{-1}\ast ^{\prime }:\Omega ^{r}\left( M\right) \rightarrow
\Omega ^{r}\left( M\right) , \\
B^{r} &:&=\ast ^{\prime }\ast ^{-1}:\Omega ^{r}\left( M\right) \rightarrow
\Omega ^{r}\left( M\right) .
\end{eqnarray*}%
Then observe that 
\begin{equation*}
A^{r}B^{r}=\ast ^{-1}\ast ^{\prime }\ast ^{\prime }\ast ^{-1}=\text{identity,%
}
\end{equation*}%
and also $B^{r}A^{r}$ is the identity. Thus we also have that 
\begin{eqnarray*}
A^{r} &=&\ast ^{-1}\ast ^{\prime }=\ast \left( \ast ^{\prime }\right) ^{-1},
\\
B^{r} &=&\ast ^{\prime }\ast ^{-1}=\left( \ast ^{\prime }\right) ^{-1}\ast .
\end{eqnarray*}%
With these definitions, 
\begin{equation*}
\ast ^{\prime }=\ast A^{r}=B^{n-r}\ast .
\end{equation*}%
Then on $\Omega ^{r}\left( M\right) $, the formal adjoint of $d$ in the $%
g^{\prime }$ metric is 
\begin{eqnarray*}
\delta ^{\prime } &=&\pm \ast A^{n-r+1}d\ast A^{r} \\
&=&B^{r-1}\delta A^{r}.
\end{eqnarray*}%
Then we check 
\begin{eqnarray*}
0 &=&\left( \delta ^{\prime }\right) ^{2}=B^{r-2}\delta A^{r-1}B^{r-1}\delta
A^{r} \\
&=&B^{r-2}\delta ^{2}A^{r}.
\end{eqnarray*}%
Consider the map on differential $r$-forms given by $\psi \mapsto \psi
^{\prime }=\left( A^{r}\right) ^{-1}\psi =B^{r}\psi $, which is an
isomorphism. Then we see that%
\begin{eqnarray*}
\delta ^{\prime }\psi ^{\prime } &=&B^{r-1}\delta A^{r}\left( A^{r}\right)
^{-1}\psi \\
&=&B^{r-1}\left( \delta \psi \right) =\left( \delta \psi \right) ^{\prime }.
\end{eqnarray*}%
Restricting now to the foliation case and the antibasic forms, we must
determine if $B^{r}$ maps the $g$-antibasic $r$-forms to the $g^{\prime }$%
-antibasic $r$ forms. We check this by taking any $g$-antibasic $r$-form $%
\psi $ and any basic form $\beta $:%
\begin{eqnarray*}
0 &=&\left\langle \beta ,\psi \right\rangle =\int_M \beta \wedge \ast \psi \\
&=&\int_M \beta \wedge \ast ^{\prime }\left( \ast
^{\prime }\right) ^{-1}\ast \psi 
=\int_M \beta \wedge \ast ^{\prime }B^{r}\psi =\left\langle \beta ,B^{r}\psi
\right\rangle ^{\prime }.
\end{eqnarray*}%
Hence $B^{r}$ maps the $g$-basic forms to the $g^{\prime }$-antibasic forms.
By the above, $\delta ^{\prime }\left( B^{r}\psi \right) =\left( \delta \psi
\right) ^{\prime }$ for antibasic $r$-forms $\psi $, so that $B^{r}\left(
\ker \delta \right) =\ker \delta ^{\prime }$ and $B^{r}\left( \func{im}%
\delta \right) =\func{im}\delta ^{\prime }$, so that the antibasic
cohomology groups corresponding to $g$ and $g^{\prime }$ are isomorphic
through the map $\left[ \psi \right] \mapsto \left[ B^{r}\psi \right]
^{\prime }$.
\end{proof}
\end{theorem}

\begin{corollary}
\label{diffeoInvtCorollary}Suppose that $\left( M,\mathcal{F}\right) $ is a
smooth foliation of a smooth, closed, oriented manifold $M$. Suppose
that $F:M\rightarrow M^{\prime }$ is a diffeomorphism, and let $\mathcal{F}%
^{\prime }$ be the foliation induced on $M^{\prime }$. Then for any two
metrics $g$, $g^{\prime }$ on $M$ and $M^{\prime }$, respectively, $%
H_{a}^{r}\left( M,\mathcal{F},g\right) \cong H_{a}^{r}\left( M^{\prime },%
\mathcal{F}^{\prime },g^{\prime }\right) $. Thus, the isomorphism class of $%
H_{a}^{r}\left( M,\mathcal{F},g\right) $ is a smooth foliation invariant.

\begin{proof}
Given the setting as above, observe that $F^{\ast}g^{\prime}$ is another metric on $%
M $. By construction and the theorem above, $H_{a}^{r}\left( M^{\prime },%
\mathcal{F}^{\prime },g^{\prime}\right) \cong H_{a}^{r}\left( M,\mathcal{F},F^{\ast
}g^{\prime}\right) \cong H_{a}^{r}\left( M,\mathcal{F},g\right) $. 
\end{proof}
\end{corollary}

\begin{notation}
Henceforth we will denote $\Omega _{a}^{r}\left( M\right) =\Omega
_{a}^{r}\left( M,g\right) $ and $H_{a}^{r}\left( M,\mathcal{F}\right) \cong
H_{a}^{r}\left( M,\mathcal{F},g\right) $, with the particular background
metric $g$ understood.
\end{notation}

\begin{lemma}
\label{qLemma}Let $\left( M,\mathcal{F}\right) $ be a smooth foliation of
codimension $q$ on a closed, oriented manifold $M$ with any Riemannian metric. Then $%
H_{a}^{k}\left( M,\mathcal{F}\right) =$ $H^{k}\left( M\right) $ for $k>q$,
and $H_{a}^{q}\left( M,\mathcal{F}\right) $ is isomorphic to a subspace of $%
H^{q}\left( M\right) $.

\begin{proof}
Since $\Omega _{a}^{k}\left( M\right) =\Omega ^{k}\left( M\right) $ for $k>q$%
, $H_{a}^{k}\left( M,\mathcal{F}\right) =$ $H^{k}\left( M\right) $ for $k>q$%
. We also have%
\begin{equation*}
H_{a}^{q}\left( M,\mathcal{F}\right) =\frac{\ker \left( \left. \delta
\right\vert _{\Omega _{a}^{q}\left( M\right) }\right) }{\func{im}\left(
\left. \delta \right\vert _{\Omega _{a}^{q+1}\left( M\right) }\right) }=%
\frac{\ker \left( \left. \delta \right\vert _{\Omega _{a}^{q}\left( M\right)
}\right) }{\func{im}\left( \left. \delta \right\vert _{\Omega ^{q+1}\left(
M\right) }\right) }\subseteq \frac{\ker \left( \left. \delta \right\vert
_{\Omega ^{q}\left( M\right) }\right) }{\func{im}\left( \left. \delta
\right\vert _{\Omega ^{q+1}\left( M\right) }\right) }=H^{q}\left( M\right) .
\end{equation*}
\end{proof}
\end{lemma}

In the case of codimension 1 foliations, we can say more.

\begin{proposition}
\label{codim1Theorem}Let $M$ be a closed, connected, oriented Riemannian manifold with
codimension $1$ foliation $\mathcal{F}$. Assume that the mean curvature form
of the normal bundle is everywhere nonzero. Then the only basic functions on 
$M$ are constants, $H_{a}^{0}\left( M\right) =\left\{ 0\right\} $, $%
H_{b}^{0}\left( M,\mathcal{F}\right) =\mathbb{R}$, and $H_{a}^{j}\left( M,%
\mathcal{F}\right) =H^{j}\left( M\right) $, $H_{b}^{j}\left( M,\mathcal{F}%
\right) =\left\{ 0\right\} $ for $j\geq 1$.

\begin{proof}
Since the normal bundle $\left( T\mathcal{F}\right) ^{\bot }$ has rank $1$,
it is involutive. Let $\nu $ be the transverse volume form of $\mathcal{F}$.
Note that $T\mathcal{F=\ker \nu }$. By Rummler's formula \cite{Rum},
\begin{equation*}
d\nu =-\kappa _{N}\wedge \nu ,
\end{equation*}%
where $\kappa _{N}$ is the mean curvature $1$-form of $\left( T\mathcal{F}%
\right) ^{\bot }$. By assumption, $\kappa _{N}$ is nonzero everywhere.
Observe that any one-form may be written $\beta =a\nu +\gamma $, where $a$
is a function and $\gamma $ is orthogonal to $\nu $. Note that a one-form $%
\beta $ is basic if and only if $X\lrcorner \beta =0$ and $X\lrcorner d\beta
=0$ for all $X\in \mathcal{\ker \nu }$. The first condition implies $\beta
=a\nu $, and the second condition implies%
\begin{eqnarray*}
0 &=&X\lrcorner \left( da\wedge \nu +ad\nu \right) =X\lrcorner \left(
da\wedge \nu -a\kappa _{N}\wedge \nu \right) \\
&=&X\lrcorner [\left( da-a\kappa _{N}\right) \wedge \nu] ,
\end{eqnarray*}%
which implies 
\begin{equation*}
X\lrcorner \left( da-a\kappa _{N}\right) =0
\end{equation*}%
for all $X\in \Gamma \left( T\mathcal{F}\right) $, or 
\begin{equation*}
da=a\kappa _{N}+b\nu
\end{equation*}%
for some function $b$. Since $\kappa _{N}\neq 0$ and is orthogonal to $\nu $%
, the maximum and minimum of the function $a$ on $M$ must occur when $a=0$,
so $a\equiv 0$. Thus, there are no nonzero basic one-forms, so that $\Omega
_{a}^{1}\left( M\right) =\Omega ^{1}\left( M\right) $. Every function $f$ on 
$M$ can be written as $f=c+\delta \alpha $ for some one-form $\alpha $ and
constant $c$ by the Hodge theorem. Since $\alpha $ and $\delta \alpha $ are
necessarily antibasic, we have the natural decomposition of $f$ into its
basic component $c$ and antibasic component $\delta \alpha $. Therefore,
every antibasic function is $\delta $-exact, and every basic function is
constant, so we have $H_{a}^{0}\left( M,\mathcal{F}\right) =\left\{
0\right\},$ $H_{b}^{0}\left( M,\mathcal{F}\right) =\mathbb{R}$, and $%
H_{a}^{j}\left( M,\mathcal{F}\right) =H^{j}\left( M\right) $, $%
H_{b}^{j}\left( M,\mathcal{F}\right) =\left\{ 0\right\} $ for $j\geq 1$,
because $\Omega _{a}^{j}\left( M\right) =\Omega ^{j}\left( M\right) $ for $%
j\geq 1$.
\end{proof}
\end{proposition}

\begin{remark}
In the next section, we consider the case of Riemannian foliations.
Codimension one Riemannian foliations always have $\kappa _{N}=0$, so the
proof of the previous proposition does not apply. Indeed, it is not true
that there are no basic one-forms, since the transverse volume form $\nu $
is always a basic one-form. Also, it is quite possible that there are
nonconstant basic functions. The cohomological facts in this case are only
different in degree 1: $H_{a}^{0}\left( M,\mathcal{F}\right) =\left\{
0\right\} $,$~H_{b}^{0}\left( M,\mathcal{F}\right) =\mathbb{R}$, $%
H_{b}^{1}\left( M,\mathcal{F}\right) =\mathbb{R}$,$~H^{1}\left( M\right)
\cong H_{a}^{1}\left( M,\mathcal{F}\right) \oplus H_{b}^{1}\left( M,\mathcal{%
F}\right) $,~and $H_{a}^{j}\left( M,\mathcal{F}\right) =H^{j}\left( M\right) 
$, $H_{b}^{j}\left( M,\mathcal{F}\right) =\left\{ 0\right\} $ for $j\geq 2$.
\end{remark}

Given two smooth foliations $(M,\mathcal{F})$ and $(M^{\prime },\mathcal{%
F^{\prime }})$, a map $f:(M,\mathcal{F})\rightarrow (M^{\prime },\mathcal{%
F^{\prime }})$ is called \textbf{foliated} if $f$ maps the leaves of $%
\mathcal{F}$ to the leaves of $\mathcal{F^{\prime }}$, which implies $%
f_{\ast }(T\mathcal{F})\subset T\mathcal{F}^{\prime }$. It follows that the
basic forms on $(M^{\prime },\mathcal{F^{\prime }})$ pull back to basic
forms on $(M,\mathcal{F})$. Two foliated maps $f,g:(M,\mathcal{F}%
)\rightarrow (M^{\prime },\mathcal{F^{\prime }})$ are \textbf{foliated
homotopic} if there exists a continuous map $H:[0,1]\times M\rightarrow
M^{\prime }$ such that $H(0,x)=f(x)$ and $\,H(1,x)=g(x)$ and for all $t\in
\lbrack 0,1]$ the map $H(t,\cdot )$ is foliated and smooth as a map from
$(M,\mathcal{F})$ to $(M',\mathcal{F}')$. A foliated map $%
f:(M,\mathcal{F})\rightarrow (M^{\prime },\mathcal{F^{\prime }})$ is a 
\textbf{foliated homotopy equivalence} if there exists a foliated map $%
h:(M^{\prime },\mathcal{F^{\prime }})\rightarrow (M,\mathcal{F})$ such that $%
f\circ h$ and $h\circ f$ are foliated homotopic to the identity on the two
foliations.

It is proved in \cite{BenRA} (also in \cite{EKN} for the case of foliated
homeomorphisms) that foliated homotopic maps induce the same map on basic
cohomology and that basic cohomology is a foliated homotopy invariant. We
now examine whether or not antibasic cohomology satisfies the same property.

Note that since in general the codifferential $\delta $ does not commute
with pullback $f^{\ast }$ by a smooth map $f:(M,\mathcal{F})\rightarrow
(M^{\prime },\mathcal{F^{\prime }})$, we do not expect that pullback induces
a linear map on antibasic cohomology. However, since it is true that on
differential forms $d\circ f^{\ast }=f^{\ast }\circ d$, we also have that 
\begin{equation*}
\left( f^{\ast }\right) ^{\dag }\circ \delta =\delta \circ \left( f^{\ast
}\right) ^{\dag },
\end{equation*}%
where $^{\dag }$ denotes the formal $L^{2}$-adjoint. Note that $f^{\ast }$
is not necessarily bounded on $L^{2}$. If we restrict to the case of closed
manifolds, $f^{\ast }$ does map smooth forms to smooth forms in $L^{2}$, so
it is a densely defined operator on $L^{2}$. Here $\left( f^{\ast }\right)
^{\dag }$ is the formal adjoint defined on its domain. From unbounded
operator theory, the domain of $\left( f^{\ast }\right) ^{\dag }$ is 
\begin{equation*}
\func{Dom}\left( \left( f^{\ast }\right) ^{\dag }\right) =\{\alpha \in
L^{2}\left( \Omega \left( M\right) \right) :\exists \gamma \in L^{2}\left(
\Omega \left( M^{\prime }\right) \right) \text{ such that }\left\langle
f^{\ast }\beta ,\alpha \right\rangle _{M}=\left\langle \beta ,\gamma
\right\rangle _{M^{\prime }}~\forall \beta \in \Omega \left( M^{\prime
}\right) \}.
\end{equation*}%
But it is known that if $\alpha $ is smooth, and if the linear map $\Phi
_{f}\left( \beta \right) :=\left\langle f^{\ast }\beta ,\alpha \right\rangle
_{M}$ is bounded, then $\Phi_f \left( \beta \right) =\left\langle \beta
,\gamma \right\rangle _{M^{\prime }}$ for some $\gamma$ by the Riesz representation theorem.
However, it turns out that $\Phi _{f}\left( \cdot \right) $ is unbounded for
almost all choices of $f$ (the rank of its differential must be constant,
for instance). In the cases where $\Phi _{f}$ is bounded, $\left(
f^{\ast }\right) ^{\dag }$ induces a linear map on antibasic cohomology. The
usual proof applies in this case to show that maps $\left( f^{\ast }\right)
^{\dag }$ are invariant over the homotopy class of such $f$.

Another possible approach is to use the Hodge star operator $\ast $ and $%
\ast ^{\prime }$ on $M$ and $M^{\prime }$, respectively, and to consider $%
\ast f^{\ast }\ast ^{\prime }$ as a map that commutes with $\delta $ up to a
sign. However, this would not apply in our case since $\ast f^{\ast }\ast
^{\prime }$ does not necessarily preserve the antibasic forms.

Thus, we still have the following open problem:

\begin{problem}
\label{OpenHomotopyQuestion}If the foliations $(M,\mathcal{F})~$and $%
(M^{\prime },\mathcal{F^{\prime }})$ with Riemannian metrics are foliated
homotopy equivalent, does that mean that their antibasic cohomology groups
are isomorphic?
\end{problem}

\begin{remark}
This problem is solved in the case of Riemannian foliations, as we see in
Theorem \ref{homotopyAxiomTheorem} and Corollary \ref%
{homotopyInvarianceCorollary}. In this case, $P_{a}$ preserves the smooth
forms, so we show that the operator $P_{a}f^{\ast }P_{a}^{\prime }$ induces
a linear map on antibasic cohomology, which is an isomorphism when $f$ is a
foliated homotopy equivalence.
\end{remark}

\section{Riemannian foliation setting\label{RiemannianFoliationSection}}

In the Riemannian foliation setting, we often restrict to basic forms. Let $%
\left( M,\mathcal{F}\right) $ be a foliation of codimension $q$ and
dimension $p$, endowed with a bundle-like metric. Again, for simplicity of exposition, we assume the foliation and manifold are
oriented.

From \cite{PaRi}, the
orthogonal projection $P_{b}:L^{2}\left( \Omega \left( M\right) \right)
\rightarrow L^{2}\left( \Omega _{b}\left( M\right) \right) $ maps smooth
forms to smooth basic forms; this was also stated and used in \cite{AL}.
Because of this, it is also true that 
\begin{equation*}
P_{a}=\left( I-P_{b}\right) :L^{2}\left( \Omega \left( M\right) \right)
\rightarrow L^{2}\left( \Omega _{b}\left( M\right) ^{\bot }\right)
\end{equation*}%
maps smooth forms to smooth \textquotedblleft antibasic
forms\textquotedblright . As described in \cite{PaRi}, we have%
\begin{equation*}
d_{b}=P_{b}dP_{b}=dP_{b},
\end{equation*}%
(i.e. $d$ restricts to the basic forms). Letting $\delta =\delta _{k}:\Omega
^{k}\left( M\right) \rightarrow \Omega ^{k-1}\left( M\right) $ be the $L^{2}$
adjoint of $d_{k-1}$, we then have%
\begin{equation*}
\delta _{b}=P_{b}\delta P_{b}=P_{b}\delta ,
\end{equation*}%
and note that the basic adjoint is $\delta _{b}=P_{b}\delta =P_{b}\delta
P_{b}$. Note also that the formulas above imply that 
\begin{eqnarray*}
\left( I-P_{b}\right) d\left( I-P_{b}\right) &=&\left( I-P_{b}\right) d, \\
\left( I-P_{b}\right) \delta \left( I-P_{b}\right) &=&\delta \left(
I-P_{b}\right) ,
\end{eqnarray*}%
or%
\begin{eqnarray}
d_{a} &=&P_{a}dP_{a}=P_{a}d,  \notag \\
\delta _{a} &=&P_{a}\delta P_{a}=\delta P_{a}.  \label{deltaPa_formula}
\end{eqnarray}%
We see 
\begin{equation*}
\delta _{a}^{2}=P_{a}\delta P_{a}P_{a}\delta P_{a}=\delta ^{2}P_{a}=0.
\end{equation*}
The adjoint of $\delta _{a}$ restricted to antibasic forms is $%
d_{a}=P_{a}dP_{a}=P_{a}d$, and again 
\begin{equation*}
d_{a}^{2}=P_{a}dP_{a}P_{a}dP_{a}=P_{a}d^{2}=0.
\end{equation*}

Also in \cite{PaRi}, it is shown that%
\begin{eqnarray}
P_{b}\delta -\delta P_{b} &=&\varepsilon \circ P_{b}=\left[ -\left(
P_{a}\kappa \right) \lrcorner +\left( -1\right) ^{p}\left( \varphi
_{0}\lrcorner \right) \left( \chi _{\mathcal{F}}\wedge \right) \right] \circ
P_{b},  \label{deltaBasicProjForm} \\
dP_{b}-P_{b}d &=&P_{b}\circ \varepsilon ^{\ast }=P_{b}\circ \left[ -\left(
P_{a}\kappa \right) \wedge +\left( -1\right) ^{p}\left( \chi _{\mathcal{F}%
}\lrcorner \right) \left( \varphi _{0}\wedge \right) \right]  \notag
\end{eqnarray}%
on $\Omega ^{\ast }\left( M\right) $. 
Recall here that 
$\chi_\mathcal{F}$ is the characteristic volume form of $T\mathcal{F}$ 
and $\varphi_0$ is a $(p+1)$-form with the property that 
$v_1\lrcorner\cdots \lrcorner v_p\lrcorner\varphi_0=0$ for any set $\{v_j\}$ 
of $p$ vectors in $T\mathcal{F}$. 
Note that $\varphi_0$ vanishes precisely when $(T\mathcal{F})^\perp$ is completely
integrable.
We observe that the only information
about the foliation needed to obtain the formulas above in \cite{PaRi} is
the fact that the orthogonal projection $P_{b}$ maps smooth forms to smooth
forms, that $P_{b}$ commutes with $\overline{\ast }$, the transversal Hodge
star-operator, and that $P_{b}\left( \alpha \wedge P_{b}\beta \right)
=\left( P_{b}\alpha \right) \wedge \left( P_{b}\beta \right) $ for all
smooth forms $\alpha ,\beta $. These facts are true for Riemannian
foliations. From the formulas above and the notation $\kappa
_{a}=P_{a}\kappa $, we obtain the following.

\begin{proposition}
\label{commutatorProposition}On an oriented Riemannian foliation $\left( M,\mathcal{F}%
\right) $ on a closed, oriented manifold with bundle-like metric,%
\begin{eqnarray}
\delta P_{a}-P_{a}\delta &=&\varepsilon \circ P_{b}=\left[ -\kappa
_{a}\lrcorner +\left( -1\right) ^{p}\left( \varphi _{0}\lrcorner \right)
\left( \chi _{\mathcal{F}}\wedge \right) \right] \circ P_{b},
\label{antiDeltaProjectionCommutator} \\
P_{a}d-dP_{a} &=&P_{b}\circ \varepsilon ^{\ast }=P_{b}\circ \left[ -\kappa
_{a}\wedge +\left( -1\right) ^{p}\left( \chi _{\mathcal{F}}\lrcorner \right)
\left( \varphi _{0}\wedge \right) \right] \text{ }
\label{anti_d_ProjectionCommutator}
\end{eqnarray}%
on $\Omega ^*\left( M\right) $. The operation $\varepsilon $ maps $%
\Omega _{b}(M) $ to $\Omega _{b}\left( M,\mathcal{F}%
\right) ^{\bot }$, and it follows that 
\begin{eqnarray}
P_{b}\varepsilon P_{b} &=&P_{b}\varepsilon ^{\ast }P_{b}=0,  \notag \\
\varepsilon P_{b} &=&P_{a}\varepsilon P_{b},~~~\varepsilon ^{\ast
}P_{b}=P_{a}\varepsilon ^{\ast }P_{b},  \notag \\
P_{b}\varepsilon P_{a} &=&P_{b}\varepsilon ,~~~P_{b}\varepsilon ^{\ast
}P_{a}=P_{b}\varepsilon ^{\ast }.  \label{epsilonProjetionFormulas}
\end{eqnarray}
\end{proposition}

\section{The antibasic Laplacian\label{firstAntibasicLaplacianSection}}

Again we assume that $\left( M,\mathcal{F}\right) $ is a foliation of
codimension $q$ and dimension $p$, endowed with a bundle-like metric, 
with orientations on both the foliation and the manifold. Recall
that the basic Laplacian is $\Delta _{b}=\delta _{b}d_{b}+d_{b}\delta _{b}=$
restriction of $P_{b}\delta d+dP_{b}\delta $ to $\Omega _{b}\left( M \right) $. We wish to do a similar restriction to antibasic
forms. Let the subscript $a$ denote the restriction to $\Omega _{a}(M)$, the antibasic forms. Then%
\begin{eqnarray*}
\Delta _{a} &=&\delta _{a}d_{a}+d_{a}\delta _{a}=\left( d_{a}+\delta
_{a}\right) ^{2} \\
&=&\text{restriction of }\delta P_{a}d+P_{a}d\delta \text{ to }\Omega _{a}(M).
\end{eqnarray*}%
From the formulas (\ref{anti_d_ProjectionCommutator}) and (\ref%
{epsilonProjetionFormulas}),%
\begin{eqnarray*}
\Delta _{a} &=&\left. (\delta P_{a}d+P_{a}d\delta )\right\vert _{\Omega _{a}(M)} \\
&=&\left.( \delta dP_{a}+\delta P_{b}\varepsilon ^{\ast }P_{a}+dP_{a}\delta
+P_{b}\varepsilon ^{\ast }P_{a}\delta )\right\vert _{\Omega _{a}(M)} \\
&=&\left. ( \delta d+\delta P_{b}\varepsilon ^{\ast }+d\delta
+P_{b}\varepsilon ^{\ast }\delta )\right\vert _{\Omega _{a}(M)} \\
&=&\left.( \Delta +\delta P_{b}\varepsilon ^{\ast }+P_{b}\varepsilon ^{\ast
}\delta )\right\vert _{\Omega _{a}(M)}.
\end{eqnarray*}%
Thus $\Delta _{a}$ is the restriction of an elliptic operator on the space
of all differential forms. Note that it is not clear whether this operator
is differential or pseudodifferential or not, since $P_{b}$ is not
pseudodifferential in general, because it is not pseudolocal. Simple examples show 
that $P_b$ can take a smooth function to a discontinuous function.

We summarize the results below.

\begin{theorem}
The antibasic Laplacian $\Delta _{a}$ satisfies the following.\label%
{antibasicLaplacianTheorem}%
\begin{equation*}
\Delta _{a}P_{a}=\widetilde{\Delta }P_{a}=P_{a}\widetilde{\Delta }^{\ast
}=P_{a}\overline{\Delta }P_{a},
\end{equation*}%
where $\widetilde{\Delta }=\Delta +\delta P_{b}\varepsilon ^{\ast
}+P_{b}\varepsilon ^{\ast }\delta $, $\widetilde{\Delta }^{\ast }=\Delta
+\varepsilon P_{b}d+d\varepsilon P_{b}$ is its adjoint, and $\overline{%
\Delta }=\Delta -\varepsilon P_{b}\varepsilon ^{\ast }$.

\begin{proof}
The first equality was shown above. To prove that $\widetilde{\Delta }%
P_{a}=P_{a}\overline{\Delta }P_{a}$, we compute%
\begin{eqnarray*}
\widetilde{\Delta }P_{a} &=&P_{a}\widetilde{\Delta }P_{a}=P_{a}\left( \Delta
+\delta P_{b}\varepsilon ^{\ast }+P_{b}\varepsilon ^{\ast }\delta \right)
P_{a} \\
&=&P_{a}\Delta P_{a}+P_{a}\delta P_{b}\varepsilon ^{\ast }P_{a} \\
&=&P_{a}\Delta P_{a}+P_{a}\left( P_{a}\delta \right) P_{b}\varepsilon ^{\ast
}P_{a} \\
&=&P_{a}\Delta P_{a}+P_{a}\left( \delta P_{a}-\varepsilon P_{b}\right)
P_{b}\varepsilon ^{\ast }P_{a} \\
&=&P_{a}\Delta P_{a}-P_{a}\varepsilon P_{b}\varepsilon ^{\ast }P_{a} \\
&=&P_{a}\left( \Delta -\varepsilon P_{b}\varepsilon ^{\ast }\right) P_{a}.
\end{eqnarray*}%
Here we have used the fact that $P_{a}P_{b}=0$, $P_a^2 =P_a$, $%
P_{b}^{2}=P_{b}$, and formula (\ref{antiDeltaProjectionCommutator}). Note
that since $P_{a}\overline{\Delta }P_{a}$ is formally self-adjoint, $%
\widetilde{\Delta }P_{a}=\left( \widetilde{\Delta }P_{a}\right) ^{\ast
}=P_{a}\widetilde{\Delta }^{\ast }$.
\end{proof}
\end{theorem}

\begin{corollary}
The antibasic Laplacian is the restriction of the ordinary Laplacian if the
mean curvature is basic and the normal bundle of the foliation is involutive.

\begin{proof}
If the mean curvature is basic and the normal bundle of the foliation is
involutive, then $P_{a}\kappa =0$ and $\varphi _{0}=0$, so that $\varepsilon
=0$. Then $\widetilde{\Delta }=\widetilde{\Delta }^{\ast }=\overline{\Delta }
$ in this case, so by the theorem above $\Delta _{a}P_{a}=\Delta P_{a}$.%
\newline
\end{proof}
\end{corollary}

Also we show a few more facts about the projections and the operators $%
d,\delta ,\varepsilon $.

\begin{proposition}
With notation as above,\label{dPlusEps*Prop}%
\begin{equation*}
P_{a}\left( d+\varepsilon ^{\ast }\right) P_{a}=\left( d+\varepsilon ^{\ast
}\right) P_{a}.
\end{equation*}

\begin{proof}
By (\ref{anti_d_ProjectionCommutator}), we have%
\begin{eqnarray*}
P_{a}\left( d+\varepsilon ^{\ast }\right) P_{a}
&=&P_{a}dP_{a}+P_{a}\varepsilon ^{\ast }P_{a} \\
&=&\left( dP_{a}+P_{b}\varepsilon ^{\ast }\right) P_{a}+P_{a}\varepsilon
^{\ast }P_{a} \\
&=&dP_{a}+\left( P_{b}+P_{a}\right) \varepsilon ^{\ast }P_{a} \\
&=&\left( d+\varepsilon ^{\ast }\right) P_{a}.
\end{eqnarray*}
\end{proof}
\end{proposition}

Now, we let the first order operator $D^{\varepsilon }$ be
defined as 
\begin{equation*}
D^{\varepsilon }=\delta +d+\varepsilon ^{\ast },
\end{equation*}%
and the antibasic operator $D_{a}^{\varepsilon }$ by%
\begin{equation*}
D_{a}^{\varepsilon }=P_{a}\left( \delta +d+\varepsilon ^{\ast }\right) P_{a}.
\end{equation*}

\begin{corollary}
\label{epsBasicDiracCor}We have 
\begin{equation*}
D_{a}^{\varepsilon }=P_{a}D^{\varepsilon }P_{a}=D^{\varepsilon }P_{a},
\end{equation*}%
so that $D_{a}^{\varepsilon }$ is the restriction of the elliptic operator $%
D^{\varepsilon }$.
\end{corollary}

\begin{corollary}
\label{epsBasicLaplCor}Let $\Delta ^{\varepsilon }:=\Delta +\varepsilon
^{\ast }\delta +\delta \varepsilon ^{\ast }$. Then $P_{a}\Delta
^{\varepsilon }P_{a}=\Delta ^{\varepsilon }P_{a}$, so that the operator $%
P_{a}\Delta ^{\varepsilon }P_{a}$ on antibasic forms is the restriction of
an elliptic operator.

\begin{proof}
By Proposition \ref{dPlusEps*Prop} and (\ref{deltaPa_formula}), we compute%
\begin{eqnarray*}
P_{a}\Delta ^{\varepsilon }P_{a} &=&P_{a}\left( \Delta +\varepsilon ^{\ast
}\delta +\delta \varepsilon ^{\ast }\right) P_{a} 
=P_{a}\left( \delta \left( d+\varepsilon ^{\ast }\right) +\left(
d+\varepsilon ^{\ast }\right) \delta \right) P_{a} \\
&=&P_{a}\delta P_{a}\left( d+\varepsilon ^{\ast }\right) P_{a}+\left(
d+\varepsilon ^{\ast }\right) P_{a}\delta P_{a} \\
&=&\delta P_{a}\left( d+\varepsilon ^{\ast }\right) P_{a}+\left(
d+\varepsilon ^{\ast }\right) \delta P_{a} \\
&=&\delta \left( d+\varepsilon ^{\ast }\right) P_{a}+\left( d+\varepsilon
^{\ast }\right) \delta P_{a}=\Delta ^{\varepsilon }P_{a}.
\end{eqnarray*}
\end{proof}
\end{corollary}

\section{Functional analysis of the antibasic de Rham and Laplace operators 
\label{FunctionalAnalysisSection}}

In this section, we show that the antibasic de Rham and Laplace operators on
a Riemannian foliation have properties similar to the ordinary de Rham and
Laplace operators on closed manifolds, namely that they have discrete
spectrum consisting of eigenvalues corresponding to finite-dimensional
eigenspaces. Note that we must work out the standard Sobolev and elliptic
theory for these operators, because in fact they are not pseudodifferential
and are not even restrictions of pseudodifferential operators to antibasic
forms.

Throughout this section, we assume that $\left( M,\mathcal{F}\right) $ is a
foliation with bundle-like metric $g$, and as in the previous section, we denote
the antibasic Hodge-de Rham and Laplace operators as 
\begin{eqnarray*}
D_{a} &=&P_{a}dP_{a}+P_{a}\delta P_{a}=d_{a}+\delta _{a}, \\
\Delta _{a} &=&d_{a}\delta _{a}+\delta _{a}d_{a}=D_{a}^{2}.
\end{eqnarray*}

First, let $H_{a}^{k}$ denote the Sobolev space $H^{k}\left( \Omega
_{a}\left( M\right) \right) $, defined as the completion of $\Omega _{a}\left(
M\right)$ with respect to a choice of the $%
k^{\text{th}}$ Sobolev norm $\left\Vert \bullet \right\Vert _{k}$; this is the same as 
the closure of $\Omega_a(M)$ inside the (complete) Sobolev space 
$H^k(\Omega(M))$.  We notate
the ordinary $L^{2}$ norm as the $0^{\,\text{th}}$ Sobolev norm $\left\Vert
\bullet \right\Vert _{0}$, so that $H_{a}^{0}=L^{2}\left( \Omega
_{a}^{k}\left( M\right) \right) $. Note that Rellich's Theorem still holds
on this subspace, i.e. the inclusion of $H_{a}^{k}\hookrightarrow
H_{a}^{\ell }$ is compact for $k>\ell $. The proof follows easily from the
standard case.

Also, note that the Sobolev embedding theorem holds for the antibasic forms,
so that for any integer $m>\frac{\dim M}{2}$, the space $H_{a}^{k+m}%
\subseteq C^{k}\left( M\right) $. This follows from the fact that $%
H_{a}^{k+m}\subseteq H^{k+m}\left( \Omega \left( M\right) \right) $.

\begin{lemma}
\label{differentialSobolevLemma}There exists a constant $c>0$ such that $%
\left\Vert D_{a}\psi \right\Vert _{0}\leq c\left\Vert \psi \right\Vert _{1}$
for all $\psi \in \Omega _{a}\left( M\right) $.

\begin{proof}
By (\ref{anti_d_ProjectionCommutator}), for any $\psi \in \Omega _{a}\left(
M\right) $,%
\begin{eqnarray*}
D_{a}\psi &=&\left( d_{a}+\delta _{a}\right) \psi =\left( P_{a}d+\delta
\right) \psi =\left( dP_{a}+P_{b}\varepsilon ^{\ast }+\delta \right) \psi \\
&=&\left( d+\delta \right) \psi +P_{b}\varepsilon ^{\ast }\psi .
\end{eqnarray*}%
Then, since $d+\delta $ is a first order differential operator and $%
\varepsilon ^{\ast }$ is a bounded operator, 
\begin{eqnarray*}
\left\Vert D_{a}\psi \right\Vert _{0} &\leq &\left\Vert \left( d+\delta
\right) \psi \right\Vert _{0}+\left\Vert P_{b}\varepsilon ^{\ast }\psi
\right\Vert _{0} \\
&\leq &c_{1}\left\Vert \psi \right\Vert _{1}+\left\Vert \varepsilon ^{\ast
}\psi \right\Vert _{0}\leq c_{1}\left\Vert \psi \right\Vert
_{1}+c_{2}\left\Vert \psi \right\Vert _{0}
\end{eqnarray*}%
for some positive constants $c_{1}$ and $c_{2}$, so that there exists $c>0$
independent of $\psi $ such that $\left\Vert D_{a}\psi \right\Vert _{0}\leq
c\left\Vert \psi \right\Vert _{1}$.
\end{proof}
\end{lemma}

\begin{lemma}
\label{GardingIneqLemma}(G\aa rding's Inequality) There exists a positive
constant $c$ such that \newline$\left\Vert \psi \right\Vert _{1}\leq c\left(
\left\Vert \psi \right\Vert _{0}+\left\Vert D_{a}\psi \right\Vert
_{0}\right) $ for all $\psi \in \Omega _{a}\left( M\right) $.

\begin{proof}
By the ordinary G\aa rding's Inequality, since $d+\delta $ is an elliptic,
first order operator on $\Omega \left( M\right) $, there exists a constant $%
c_{1}$ such that for all $\psi \in \Omega _{a}\left( M\right) \subseteq
\Omega \left( M\right) $, 
\begin{equation*}
\left\Vert \psi \right\Vert _{1}\leq c_{1}\left( \left\Vert \psi \right\Vert
_{0}+\left\Vert \left( d+\delta \right) \psi \right\Vert _{0}\right) .
\end{equation*}%
Then, again by (\ref{anti_d_ProjectionCommutator}) and the proof of
Lemma \ref{differentialSobolevLemma},
\begin{eqnarray*}
\left\Vert \psi \right\Vert _{1} &\leq &c_{1}\left( \left\Vert \psi
\right\Vert _{0}+\left\Vert \left( d_{a}+\delta _{a}-P_{b}\varepsilon ^{\ast
}\right) \psi \right\Vert _{0}\right) \\
&\leq &c_{1}\left( \left\Vert \psi \right\Vert _{0}+\left\Vert
P_{b}\varepsilon ^{\ast }\psi \right\Vert _{0}+\left\Vert \left(
d_{a}+\delta _{a}\right) \psi \right\Vert _{0}\right) \\
&\leq &c_{1}\left( \left\Vert \psi \right\Vert _{0}+\left\Vert \varepsilon
^{\ast }\psi \right\Vert _{0}+\left\Vert \left( d_{a}+\delta _{a}\right)
\psi \right\Vert _{0}\right) ,
\end{eqnarray*}%
so since $\varepsilon ^{\ast }$ is bounded, the result follows.
\end{proof}
\end{lemma}

\begin{lemma}
\label{TechnicalLemma0}For all nonnegative integers $k$, there exists a
positive constant $c_{k}$ such that 
\begin{equation*}
\left\Vert P_{a}\phi \right\Vert _{k}\leq c_{k}\left\Vert \phi \right\Vert
_{k}\text{ and }\left\Vert P_{b}\phi \right\Vert _{k}\leq c_{k}\left\Vert
\phi \right\Vert _{k}
\end{equation*}%
for any differential form $\phi $.

\begin{proof}
We use induction on $k$. Let $\phi $ be any differential form. Observe first that $\left\Vert P_{a}\phi \right\Vert _{0}\leq \left\Vert
\phi \right\Vert _{0},\left\Vert P_{b}\phi \right\Vert _{0}\leq \left\Vert
\phi \right\Vert _{0}$. Next, suppose that the results have been shown for
some nonnegative integer $k$. Since $D^{\varepsilon }$ is elliptic on all
forms, it satisfies the ordinary elliptic estimates: there exist constants $%
b_{1}$ and $b_{2}$ such that 
\begin{eqnarray*}
\left\Vert P_{a}\phi \right\Vert _{k+1} &\leq &b_{1}\left\Vert
D^{\varepsilon }P_{a}\phi \right\Vert _{k}+b_{2}\left\Vert P_{a}\phi
\right\Vert _{k} \\
&=&b_{1}\left\Vert \left( d+\delta +\varepsilon ^{\ast }\right) P_{a}\phi
\right\Vert _{k}+b_{2}\left\Vert P_{a}\phi \right\Vert _{k} \\
&=&b_{1}\left\Vert \left( P_{a}d-P_{b}\varepsilon ^{\ast }+P_{a}\delta
+\varepsilon P_{b}+\varepsilon ^{\ast }P_{a}\right) \phi \right\Vert
_{k}+b_{2}\left\Vert P_{a}\phi \right\Vert _{k} \\
&=&b_{1}\left\Vert \left( P_{a}d-P_{b}\varepsilon ^{\ast }P_{a}+P_{a}\delta
+\varepsilon P_{b}+\varepsilon ^{\ast }P_{a}\right) \phi \right\Vert
_{k}+b_{2}\left\Vert P_{a}\phi \right\Vert _{k} \\
&=&b_{1}\left\Vert \left( P_{a}d+P_{a}\delta +\varepsilon
P_{b}+P_{a}\varepsilon ^{\ast }P_{a}\right) \phi \right\Vert
_{k}+b_{2}\left\Vert P_{a}\phi \right\Vert _{k} \\
&\leq &b_{1}\left( \left\Vert P_{a}\left( d+\delta \right) \phi \right\Vert
_{k}+\left\Vert \varepsilon P_{b}\phi \right\Vert _{k}+\left\Vert
P_{a}\varepsilon ^{\ast }P_{a}\phi \right\Vert _{k}\right) +b_{2}\left\Vert
P_{a}\phi \right\Vert _{k}.
\end{eqnarray*}%
Using the fact that $\varepsilon $ is a zeroth order differential operator
and the induction hypothesis,%
\begin{eqnarray*}
\left\Vert P_{a}\phi \right\Vert _{k+1} &\leq &\left( \text{constant}\right)
\left\Vert \left( d+\delta \right) \phi \right\Vert _{k}+\left( \text{%
constant}\right) \left\Vert \phi \right\Vert _{k} \\
&\leq &\left( \text{constant}\right) \left\Vert \phi \right\Vert
_{k+1}+\left( \text{constant}\right) \left\Vert \phi \right\Vert _{k}\leq
\left( \text{constant}\right) \left\Vert \phi \right\Vert _{k+1},
\end{eqnarray*}%
since $d+\delta $ is a first order operator. Also, 
\begin{equation*}
\left\Vert P_{b}\phi \right\Vert _{k+1}=\left\Vert \phi -P_{a}\phi
\right\Vert _{k+1}\leq \left\Vert \phi \right\Vert _{k+1}+\left\Vert
P_{a}\phi \right\Vert _{k+1}\leq \left( \text{constant}\right) \left\Vert
\phi \right\Vert _{k+1}.
\end{equation*}%
By induction, the proof is complete.
\end{proof}
\end{lemma}

\begin{lemma}
\label{technLemma}Let $D^{\varepsilon }=d+\delta +\varepsilon ^{\ast }$ as
an operator on all differential forms, and let $D_{a}=d_{a}+\delta _{a}$ be
the antibasic de Rham operator. For all nonnegative integers $k$, there
exists a positive constant $c_{k}$ such that 
\begin{equation*}
\left\Vert D^{\varepsilon }\psi \right\Vert _{k}-c_{k}\left\Vert \psi
\right\Vert _{k}\leq \left\Vert D_{a}\psi \right\Vert _{k}\leq \left\Vert
D^{\varepsilon }\psi \right\Vert _{k}+c_{k}\left\Vert \psi \right\Vert _{k},
\end{equation*}%
for any antibasic form $\psi $.

\begin{proof}
For any antibasic $\psi $, \newline
\begin{equation*}
\left\Vert D_{a}\psi \right\Vert _{k}=\left\Vert D^{\varepsilon }\psi
+\left( D^{a}-D^{\varepsilon }\right) \psi \right\Vert _{k}=\left\Vert
D^{\varepsilon }\psi +P_{a}\varepsilon ^{\ast }\psi \right\Vert _{k}.
\end{equation*}%
\newline
It suffices to bound 
$\left\Vert P_{a}\varepsilon ^{\ast }\psi \right\Vert
_{k}$. This follows from a bound on $\left\Vert P_{a}\phi \right\Vert
_{k}$ 
from Lemma \ref{TechnicalLemma0}, since $\varepsilon ^{\ast }$ is a
zeroth order operator.
\end{proof}
\end{lemma}

\begin{lemma}
(Elliptic Estimates for $D_{a}$)\label{ellipticEstimatesLemma} For every
integer $k\geq 0$, there exists a positive constant $C_{k}$ such that $%
\left\Vert \psi \right\Vert _{k+1}\leq C_{k}\left( \left\Vert \psi
\right\Vert _{k}+\left\Vert D_{a}\psi \right\Vert _{k}\right) $ for all $%
\psi \in \Omega _{a}\left( M\right) $.

\begin{proof}
Let $k$ be a nonnegative integer. From the elliptic estimates for the
operator $D^{\varepsilon }$ on all forms, there exists a positive constant $%
b_{k}$ such that for any $\psi \in \Omega _{a}\left( M\right) $, 
\begin{eqnarray*}
\left\Vert \psi \right\Vert _{k+1} &\leq &b_{k}\left( \left\Vert \psi
\right\Vert _{k}+\left\Vert D^{\varepsilon }\psi \right\Vert _{k}\right) \\
&\leq &b_{k}\left( \left\Vert \psi \right\Vert _{k}+\left\Vert D_{a}\psi
\right\Vert _{k}+c_{k}\left\Vert \psi \right\Vert _{k}\right)
\end{eqnarray*}%
for a positive constant $c_{k}$, by Lemma \ref{technLemma}. The inequality
follows by letting $C_{k}={\rm max}\left(b_k, b_{k}\left( 1+c_{k}\right)\right) $.
\end{proof}
\end{lemma}

\begin{remark}
The case $k=0$ is G\aa rding's Inequality, which we have shown independently
in Lemma \ref{GardingIneqLemma}.
\end{remark}

\begin{lemma}
\label{D_a_formallySA}The operator $D_{a}$ on $\Omega _{a}\left( M\right) $
is formally self-adjoint.

\begin{proof}
For any antibasic forms $\alpha $ and $\beta $,%
\begin{eqnarray*}
\left\langle D_{a}\alpha ,\beta \right\rangle &=&\left\langle \left(
P_{a}\left( d+\delta \right) P_{a}\right) \alpha ,\beta \right\rangle \\
&=&\left\langle \alpha ,P_{a}\left( d+\delta \right) P_{a}\beta
\right\rangle =\left\langle \alpha ,D_{a}\beta \right\rangle .
\end{eqnarray*}
\end{proof}
\end{lemma}

\begin{lemma}
The domain of the closure of $D_{a}$ is $H_{a}^{1}$.

\begin{proof}
The graph of $D_{a}$ is $G_{a}=\left\{ \left( \omega ,D_{a}\omega \right)
:\omega \in \Omega _{a}\left( M\right) \right\} \subseteq H_{a}^{0}\times
H_{a}^{0}$. The closure of $G_{a}$ is also a graph, by the following
argument. We must show that for any $\left( 0,\eta \right) \in \overline{%
G_{a}}$, $\eta =0$. For any $\left( 0,\eta \right) \in \overline{G_{a}}$,
there is a sequence $\left( \omega _{j}\right) $ of smooth antibasic forms
with $\omega _{j}\rightarrow 0$ and $D_{a}\omega _{j}\rightarrow \eta $ in $%
H_{a}^{0}\subseteq L^{2}$. But then for any smooth antibasic form $\gamma $, 
\begin{equation*}
\left\langle D_{a}\omega _{j},\gamma \right\rangle \rightarrow \left\langle
\eta ,\gamma \right\rangle \text{, and }\left\langle \omega _{j},D_{a}\gamma
\right\rangle \rightarrow 0
\end{equation*}%
as $j\rightarrow \infty $. But $\left\langle \omega _{j},D_{a}\gamma
\right\rangle =\left\langle D_{a}\omega _{j},\gamma \right\rangle $ by Lemma %
\ref{D_a_formallySA}, so $\left\langle \eta ,\gamma \right\rangle =0$ for
all smooth $\gamma $, so $\eta =0$. Thus $\overline{G_{a}}=\left\{ \left( \omega ,A\omega \right) :\omega \in 
\func{dom}\left( A\right) \right\} $ for some operator $A$, which is defined
to be the closure of $D_{a}$. Thus the domain is the set of all $\omega \in
H_{a}^{0}$ such that there exists a sequence $\left( \omega _{j}\right) $ of
smooth antibasic forms such that $\omega _{j}\rightarrow \omega $ in $%
H_{a}^{0}$ and $\left( D_{a}\omega _{j}\right) $ converges in $H_{a}^{0}$.
By G\aa rding's Inequality (Lemma \ref{GardingIneqLemma}) and Lemma \ref%
{differentialSobolevLemma}, $\func{dom}\left( A\right) =H_{a}^{1}$.
\end{proof}
\end{lemma}

\begin{lemma}
\label{FM_Lemma}(Existence of Friedrichs' mollifiers) There exists a family
of self-adjoint smoothing operators $\left\{ F_{\rho }\right\} _{\rho \in
(0,1)}$ on $H_{a}^{0}$ such that $\left( F_{\rho }\right) $ is bounded in $%
H_{a}^{0}$, $F_{\rho }\rightarrow 1$ uniformly weakly in $H_{a}^{0}$ as $%
\rho \rightarrow 0$, and $[F_{\rho },D_{a}]$ extends to a uniformly bounded
family of operators on $H_{a}^{0}$.

\begin{proof}
Let $F_{\rho }^{0}$ be defined as the usual Friedrichs' mollifiers; c.f. 
\cite[Definition 5.21, Exercise 5.34]{Roe}. Thus, these operators satisfy
the properties above, except with $H_{a}^{0}$ replaced by $L^{2}\left(
\Omega \left( M\right) \right) $ and $D_{a}$ replaced by any first order
differential operator. Now let $F_{\rho }=P_{a}F_{\rho }^{0}$. Note that $%
F_{\rho }$ is smoothing because $F_{\rho }^{0}$ is smoothing and since $%
P_{a} $ maps smooth forms to smooth forms; its kernel is the kernel of $%
F_{\rho }^{0}$ followed by $P_{a}$. We now check the three properties.
First, for any $\alpha \in H_{a}^{0}\subseteq L^{2}\left( \Omega \left(
M\right) \right) $, 
\begin{equation*}
\left\Vert F_{\rho }\alpha \right\Vert _{0}=\left\Vert P_{a}F_{\rho
}^{0}\alpha \right\Vert _{0}\leq \left\Vert F_{\rho }^{0}\alpha \right\Vert
_{0}\leq c\left\Vert \alpha \right\Vert _{0}
\end{equation*}%
for some $c>0$, by the first property of $F_{\rho }^{0}$. Next, for any $%
\alpha \in H_{a}^{0}$, for all smooth antibasic forms $\beta $,%
\begin{equation*}
\left\langle \left( F_{\rho }-1\right) \alpha ,\beta \right\rangle
=\left\langle \left( P_{a}F_{\rho }^{0}-1\right) \alpha ,\beta \right\rangle
=\left\langle \left( F_{\rho }^{0}-1\right) \alpha ,\beta \right\rangle ,
\end{equation*}%
which approaches $0$ uniformly as $\rho \rightarrow 0$ by the corresponding
property of $F_{\rho }^{0}$. Lastly, for any smooth antibasic forms $\omega
,\eta $, 
\begin{eqnarray*}
\left\langle \lbrack F_{\rho },D_{a}]\omega ,\eta \right\rangle
&=&\left\langle P_{a}F_{\rho }^{0}D_{a}\omega ,\eta \right\rangle
-\left\langle D_{a}P_{a}F_{\rho }^{0}\omega ,\eta \right\rangle \\
&=&\left\langle F_{\rho }^{0}\left( \delta +d+P_{b}\varepsilon ^{\ast
}\right) \omega ,\eta \right\rangle -\left\langle P_{a}\left( \delta
+d+\varepsilon P_{b}\right) F_{\rho }^{0}\omega ,\eta \right\rangle \\
&=&\left\langle F_{\rho }^{0}\left( \delta +d\right) \omega ,\eta
\right\rangle +\left\langle F_{\rho }^{0}P_{b}\varepsilon ^{\ast }\omega
,\eta \right\rangle -\left\langle \left( \delta +d+\varepsilon P_{b}\right)
F_{\rho }^{0}\omega ,\eta \right\rangle \\
&=&\left\langle \left[ F_{\rho }^{0},\left( \delta +d\right) \right] \omega
,\eta \right\rangle +\left\langle F_{\rho }^{0}P_{b}\varepsilon ^{\ast
}\omega ,\eta \right\rangle -\left\langle \varepsilon P_{b}F_{\rho
}^{0}\omega ,\eta \right\rangle .
\end{eqnarray*}%
The operator $\left[ F_{\rho }^{0},\left( \delta +d\right) \right] $ is
bounded by the corresponding property of $F_{\rho }^{0}$, and $F_{\rho
}^{0}P_{b}\varepsilon ^{\ast }$ and $\varepsilon P_{b}F_{\rho }^{0}$ are
both zeroth order operators that are uniformly bounded in $\rho $ on $L^{2}$%
, so we conclude that $[F_{\rho },D_{a}]$ is uniformly bounded on $H_{a}^{0}$%
.
\end{proof}
\end{lemma}

\begin{corollary}\label{uniformelyboundedcommutator}
Let $\left\{ F_{\rho }\right\} $ be a family of Friedrichs' mollifiers. Then 
$F_{\rho }$ and $\left[ D_{a},F_{\rho }\right] $ are uniformly bounded
families of operators on $H_{a}^{k}$ for any $k\geq 0$.

\begin{proof}
We proceed by induction using the elliptic estimates in Lemma \ref{ellipticEstimatesLemma}.
\end{proof}
\end{corollary}

\begin{proposition}
\label{ellipticSolutionProposition}Suppose that $\alpha ,\beta \in H_{a}^{0}$
and $D_{a}\alpha =\beta $ weakly. Then $\alpha \in H_{a}^{1}=\func{dom}%
\overline{D_{a}}$, and $\overline{D_{a}}\alpha =\beta $.

\begin{proof}
For any $\alpha ,\beta \in H_{a}^{0}$, suppose $D_{a}\alpha =\beta $ weakly.
Then for any smooth antibasic form $\gamma $ and any $\rho \in (0,1),$ 
\begin{eqnarray*}
\left\langle D_{a}F_{\rho }\alpha ,\gamma \right\rangle &=&\left\langle
F_{\rho }\alpha ,D_{a}\gamma \right\rangle \\
&=&\left\langle \alpha ,F_{\rho }D_{a}\gamma \right\rangle =\left\langle
\alpha ,D_{a}F_{\rho }\gamma \right\rangle +\left\langle \alpha ,[F_{\rho
},D_{a}]\gamma \right\rangle \\
&=&\left\langle \beta ,F_{\rho }\gamma \right\rangle +\left\langle \alpha
,[F_{\rho },D_{a}]\gamma \right\rangle \\
&=&\left\langle F_{\rho }\beta ,\gamma \right\rangle +\left\langle \alpha
,[F_{\rho },D_{a}]\gamma \right\rangle
\end{eqnarray*}%
by Lemma \ref{FM_Lemma}. Thus, for a constant $C>0$ independent of $\rho $
and $\gamma $, 
\begin{equation*}
\left\vert \left\langle D_{a}F_{\rho }\alpha ,\gamma \right\rangle
\right\vert \leq C\left\Vert \gamma \right\Vert _{0}.
\end{equation*}%
Then $\left\Vert D_{a}F_{\rho
}\alpha \right\Vert _{0}\leq C$. By G\aa rding's Inequality (Lemma \ref%
{GardingIneqLemma}) and the fact that $F_{\rho }$ is a bounded operator in $%
H_{a}^{0}$, $\left\{ F_{\rho }\alpha \right\} _{\rho \in (0,1)}$ is a
bounded set in $H_{a}^{1}$. By the weak compactness of a ball in the Hilbert
space $H_{a}^{1}$ (with equivalent metric $\left\langle \xi ,\theta
\right\rangle _{1}=\left\langle D_{a}\xi ,D_{a}\theta \right\rangle
+\left\langle \xi ,\theta \right\rangle $), there is a sequence $\rho
_{j}\rightarrow 0$ and $\alpha ^{\prime }\in H_{a}^{1}$ such that $F_{\rho
_{j}}\alpha \rightarrow \alpha ^{\prime }$ weakly in $H_{a}^{1}$. By
Rellich's Theorem, the subsequence converges strongly in $H_{a}^{0}$, so $%
F_{\rho _{j}}\alpha \rightarrow \alpha ^{\prime }$ in $H_{a}^{0}$. But we
know already that $F_{\rho _{j}}\alpha \rightarrow \alpha $ in $H_{a}^{0}$,
so $\alpha =\alpha ^{\prime }\in H_{a}^{1}$.\newline
\end{proof}
\end{proposition}

\begin{corollary}
\label{Essential Self-Adjointness Corollary}The antibasic de Rham and
antibasic Laplacian are essentially self-adjoint operators.

\begin{proof}
From the proposition above, the domain of the closure of the symmetric
operator $D_{a}$ is the domain of its $H_{a}^{0}$-adjoint, so that $%
\overline{D_{a}}$ is self-adjoint. Then $\Delta _{a}=D_{a}^{2}$ is also
essentially self-adjoint.
\end{proof}
\end{corollary}

\begin{proposition}
\label{Elliptic Regularity Proposition}(Elliptic regularity) Suppose that $%
\omega \in \ker D_{a}\subseteq H_{a}^{1}$. Then $\omega $ is smooth.

\begin{proof}
If $D_{a}\omega =0$ for some $\omega \in H_{a}^{1}$. We will show by
induction that $\omega \in H_{a}^{k}$ for all $k$, and then the Sobolev
embedding theorem implies that $\omega $ is smooth. Suppose that we know $%
\omega \in H_{a}^{k-1}$ for some $k\geq 2$. Let $\left\{ F_{\rho }\right\} $
be a family of Friedrich's mollifiers. Then from the elliptic estimates
(Lemma \ref{ellipticEstimatesLemma}), there is a constant $C_{k-1}>0$ such that
\begin{eqnarray*}
\left\Vert F_{\rho }\omega \right\Vert _{k} &\leq &C_{k-1}\left( \left\Vert
F_{\rho }\omega \right\Vert _{k-1}+\left\Vert D_{a}F_{\rho }\omega
\right\Vert _{k-1}\right) \\
&\leq &C_{k-1}\left( \left\Vert F_{\rho }\omega \right\Vert
_{k-1}+\left\Vert F_{\rho }D_{a}\omega \right\Vert _{k-1}+\left\Vert \left[
D_{a},F_{\rho }\right] \omega \right\Vert _{k-1}\right) \\
&=&C_{k-1}\left( \left\Vert F_{\rho }\omega \right\Vert _{k-1}+\left\Vert 
\left[ D_{a},F_{\rho }\right] \omega \right\Vert _{k-1}\right) .
\end{eqnarray*}%
Thus $\left\Vert F_{\rho }\omega \right\Vert _{k}$ is bounded by Corollary \ref{uniformelyboundedcommutator}. We now proceed as in the proof of Proposition \ref{ellipticSolutionProposition} to say that there is a sequence $\rho_j\to 0$ such that $F_{\rho _{j}}\omega \rightarrow \omega' $ weakly in $H_{a}^{k}$ and strongly to $H_{a}^{0}$. Thus, we get $\omega=\omega' \in H_{a}^{k}$.
\end{proof}
\end{proposition}

\begin{corollary}
Eigenforms of $D_{a}$ are smooth.

\begin{proof}
The proof above also is easily modified if $D_{a}\omega =\lambda \omega $ to
show that the eigenforms of $D_{a}$ are smooth.
\end{proof}
\end{corollary}

We will now use a standard technique to derive the spectral theorem for $%
D_{a}$ and $\Delta _{a}$ from these basic facts (c.f. \cite[Chapter 5]{Roe})

\begin{lemma}
\label{orthogonalDirectSumLemma}Let $\overline{G}=
\{ (\omega,D_a\omega) : \omega\in H_a^1\} \subseteq H_a^1\times H_a^0$
 denote the closure of the graph of $D_{a}$. Let 
$J:H_{a}^{0}\times H_{a}^{0}\rightarrow H_{a}^{0}\times H_{a}^{0}$ be
defined by $J\left( x,y\right) =\left( y,-x\right) $. Then there is an
orthogonal direct sum decomposition 
\begin{equation*}
H_{a}^{0}\oplus H_{a}^{0}=\overline{G}\oplus J\overline{G} .
\end{equation*}

\begin{proof}
Suppose $\left( x,y\right) \in \overline{G}^{\bot }$. Then, for all $\omega
\in \Omega _{a}\left( M\right) $, 
\begin{equation*}
0=\left\langle \left( x,y\right) ,\left( \omega ,D_{a}\omega \right)
\right\rangle =\left\langle x,\omega \right\rangle +\left\langle
y,D_{a}\omega \right\rangle ,
\end{equation*}%
so that $D_{a}y+x=0$ weakly. By Proposition \ref{ellipticSolutionProposition}%
, $y\in H_{a}^{1}$, so $\left( y,-x\right) \in \overline{G}$, so $\left(
x,y\right) \in J\overline{G}$.
\end{proof}
\end{lemma}

\begin{definition}
Let the operator $Q_{a}:H_{a}^{0}\rightarrow H_{a}^{1}$ be defined by the following
equation: for any $\alpha \in H_{a}^{0}$, $(Q_{a}\alpha ,D_{a}Q_{a}\alpha )$
is the orthogonal projection of $\left( \alpha ,0\right) $ to $\overline{G}$
in $H_{a}^{0}\oplus H_{a}^{0}$.
\end{definition}

As $\left\Vert \alpha \right\Vert _{0}^{2}\geq \left\Vert Q_{a}\alpha
\right\Vert _{0}^{2}+\left\Vert D_{a}Q_{a}\alpha \right\Vert _{0}^{2}\geq
c\Vert Q_{a}\alpha \Vert _{1}^{2}$, then $Q_{a}$ is bounded as an operator
from $H_{a}^{0}$ to $H_{a}^{1}$. By Rellich's Theorem, $Q_{a}$ is compact as
an operator from $H_{a}^{0}$ to $H_{a}^{0}$. It is self-adjoint, positive,
and injective, and has norm $\leq 1$. By the spectral theorem for compact,
self-adjoint operators, $H_{a}^{0}$ can be decomposed as a direct sum of
finite-dimensional eigenspaces of $Q_{a}$, and the eigenvalues approach $0$
as the only accumulation point. Given an eigenvector $\alpha $ of $Q_{a}$
corresponding to the eigenvalue $\mu >0$, so that $0<\mu \leq 1$, by Lemma %
\ref{orthogonalDirectSumLemma} there exists $\eta $ such that%
\begin{eqnarray*}
\left( \alpha ,0\right) &=&\left( Q_{a}\alpha ,D_{a}Q_{a}\alpha \right)
+\left( -D_{a}\eta ,\eta \right) \\
&=&\mu \left( \alpha ,D_{a}\alpha \right) +\left( -D_{a}\eta ,\eta \right) ,
\end{eqnarray*}%
so that $\left( \mu -1\right) \alpha =D_{a}\eta $ and $\eta =-\mu
D_{a}\alpha $. Letting $\lambda ^{2}=\frac{1-\mu }{\mu }$ and $\beta =-\frac{%
1}{\mu \lambda }\eta $, we have%
\begin{equation*}
D_{a}\alpha =\lambda \beta ,~D_{a}\beta =\lambda \alpha.
\end{equation*}%
Thus $\alpha \pm \beta $ are eigenforms of $D_{a}$ with eigenvalues $\pm
\lambda $. Thus, $H_{a}^{0}$ can be decomposed as an orthogonal direct sum
of finite-dimensional eigenspaces of $D_{a}$. We now have the following.

\begin{theorem}
\label{spectralTheorem}(Spectral Theorem for the antibasic operators) The
spectrum of the antibasic Laplacian and antibasic de Rham operators consists
of real eigenvalues of finite multiplicity, with accumulation points at
infinity. The smooth eigenforms of $D_{a}$ are also the eigenforms of $%
\Delta _{a}$ and can be chosen to form a complete orthonormal basis of $%
H_{a}^{0}$.

\begin{proof}
Besides the above computations, observe that $\Delta _{a}=D_{a}^{2}$.
\end{proof}
\end{theorem}

\section{The antibasic Hodge decomposition and homotopy invariance\label%
{antibasicHodgeSection}}

Let $M$ be a closed manifold of dimension $n$ endowed with a foliation of
codimension $q$ and a bundle-like metric. The basic Hodge decomposition
theorem (proved in \cite{PaRi}) gives 
\begin{equation*}
\Omega _{b}^{k}\left( M\right) =\func{im}\left( d_{b,k-1}\right) \oplus 
\mathcal{H}_{b}^{k}\oplus \func{im}\left( \delta _{b,k+1}\right) ,
\end{equation*}%
where $d_{b,k}=d:\Omega _{b}^{k}(M)\rightarrow \Omega _{b}^{k+1}(M)$ is the
exterior derivative restricted to basic forms with $L^{2}$-adjoint $\delta
_{b,k+1}=P_{b}\delta :\Omega _{b}^{k+1}(M)\rightarrow \Omega _{b}^{k}(M)$, and
where $\mathcal{H}_{b}^{k}=\ker \left( \Delta _{b,k}\right) $ is the space
of basic harmonic $k$-forms. Also%
\begin{equation*}
\ker \left( d_{b,k}\right) =\func{im}\left( d_{b,k-1}\right) \oplus \mathcal{%
H}_{b}^{k}\text{ and }\ker \left( \delta _{b,k}\right) =\mathcal{H}%
_{b}^{k}\oplus \func{im}\left( \delta _{b,k+1}\right) ,
\end{equation*}%
so the basic cohomology groups satisfy $H_{b}^{k}\left( M,\mathcal{F}\right)
\cong \mathcal{H}_{b}^{k}$. We now have the tools to prove the antibasic
version.

First, note that there is an alternative de Rham complex that uses $\delta $
as a differential. Writing $\Omega ^{j}=\Omega ^{j}\left( M\right) $, the
complex 
\begin{equation*}
\Omega ^{n}\overset{\delta _{n}}{\longrightarrow }\Omega ^{n-1}\overset{%
\delta _{n-1}}{\longrightarrow }...\overset{\delta _{1}}{\longrightarrow }%
\Omega ^{0}\longrightarrow 0
\end{equation*}%
satisfies $\delta _{k-1}\delta _{k}=0$ for $0\leq k\leq n$, and the de Rham
cohomology satisfies%
\begin{equation*}
H^{k}\left( M\right) =\frac{\ker\left( \delta _{k}:\Omega ^{k}\rightarrow \Omega
^{k-1}\right)}{\func{im}\left(\delta _{k+1}:\Omega ^{k-1}\rightarrow \Omega ^{k}\right)}.
\end{equation*}%
Abbreviating $\Omega _{a}^{j}=\Omega _{a}^{j}\left( M\right) $, the
antibasic de Rham complex is a subcomplex 
\begin{equation*}
\Omega _{a}^{n}\overset{\delta _{n}}{\longrightarrow }\Omega _{a}^{n-1}%
\overset{\delta _{n-1}}{\longrightarrow }...\overset{\delta _{1}}{%
\longrightarrow }\Omega _{a}^{0}\longrightarrow 0
\end{equation*}%
We adopt a standard proof of Hodge decomposition to our case (c.f. \cite[%
Chapter 6]{Roe}) and utilize the analytic results proved in the previous
section.

\begin{theorem}
\label{AnitbasicHodgeTheorem}(\textbf{Antibasic Hodge Theorem}) Suppose that 
$\left( M,\mathcal{F}\right) $ is a Riemannian foliation with bundle-like
metric. Then for $0\leq k\leq n$ the antibasic cohomology groups satisfy%
\begin{equation*}
H_{a}^{k}\left( M,\mathcal{F}\right) \cong \mathcal{H}_{a}^{k}.
\end{equation*}

\begin{proof}
From Theorem \ref{spectralTheorem}, $\mathcal{H}_{a}^{k}=\ker \left( \left.
\Delta _{a}\right\vert _{\Omega _{a}^{k}}\right) $ is finite dimensional for
all $k$. Consider the following subcomplex of the antibasic de Rham complex,
with $0$ being the codifferential, and the inclusion maps:%
\begin{equation*}
\begin{array}{ccccccccc}
... & \overset{0}{\longrightarrow } & \mathcal{H}_{a}^{j+1} & \overset{0}{%
\longrightarrow } & \mathcal{H}_{a}^{j} & \overset{0}{\longrightarrow } & 
\mathcal{H}_{a}^{j-1} & \overset{0}{\longrightarrow } & ... \\ 
&  & \downarrow ^{\imath } &  & \downarrow ^{\imath } &  & \downarrow
^{\imath } &  &  \\ 
... & \overset{\delta _{j+2}}{\longrightarrow } & \Omega _{a}^{j+1} & 
\overset{\delta _{j+1}}{\longrightarrow } & \Omega _{a}^{j} & \overset{%
\delta _{j}}{\longrightarrow } & \Omega _{a}^{j-1} & \overset{\delta _{j-1}}{%
\longrightarrow } & ...%
\end{array}%
\end{equation*}%
\newline
We will show that the inclusion $\imath $ is a chain equivalence. We define
the map $P:\Omega _{a}^{j}\rightarrow \mathcal{H}_{a}^{j}$ to be the
restriction of the orthogonal projection $L^{2}\left( \Omega _{a}^{j}\right)
\rightarrow \mathcal{H}_{a}^{j}$ to smooth antibasic forms. Then $P\imath =1$
and $\imath P=1-f\left( D_{a}\right) $, where 
\begin{equation*}
f\left( \lambda \right) =\left\{ 
\begin{array}{ll}
1~ & \text{if }\lambda \neq 0 \\ 
0~ & \text{if }\lambda =0%
\end{array}%
\right.
\end{equation*}%
and where we have used Theorem \ref{spectralTheorem} and the functional
calculus. Let 
\begin{equation*}
g\left( \lambda \right) =\left\{ 
\begin{array}{ll}
\frac{1}{\lambda ^{2}}~ & \text{if }\lambda \neq 0 \\ 
0~ & \text{if }\lambda =0%
\end{array}%
\right. .
\end{equation*}%
Then $g$ is bounded on $\sigma \left( D_{a}\right) $, so the Green's
operator $G_{a}=g\left( D_{a}\right) $ extends to a bounded operator on $%
H_{a}^{0}$. We see that $D_{a}^{2}G_{a}=f\left( D_{a}\right) =1-\imath P$,
and also 
\begin{equation*}
D_{a}^{2}G_{a}=\left( \delta d_{a}+d_{a}\delta \right) G_{a}=\delta
d_{a}G_{a}+G_{a}d_{a}\delta .
\end{equation*}%
Since $\Delta _{a}$ commutes with $d_{a}$, we have $%
H_{a}=d_{a}G_{a}=G_{a}d_{a}$. Thus, $H_{a}$ satisfies%
\begin{equation*}
1-\imath P=\delta H_{a}+H_{a}\delta
\end{equation*}%
and is thus a chain homotopy between $\imath P$ and $1$, so $\imath $ is a
chain equivalence. \newline
\end{proof}
\end{theorem}

\begin{corollary}
On a Riemannian foliation on a closed manifold, the antibasic cohomology
groups are finite dimensional.
\end{corollary}

\begin{remark}
After reading an early version of this paper, J.A. \'Alvarez-L\'opez observed that when $\mathcal{F}$  is Riemannian, 
the Hodge star operator maps the antibasic complex isomorphically to
the $p$-basic complex investigated in \cite{Ser}, where the author showed that every $k$-basic cohomology group is
finite dimensional; the corollary also follows from that observation.
If $p = 1$, $M$ is
oriented, and $\mathcal F$ is Riemannian, then $H^\bullet_a(M,\mathcal{F})$
is in fact isomorphic to the term $E_2^{1,\bullet}$ of the spectral sequence of $\mathcal{F}$;
see \cite[Lemma 2.5]{Ser}. This agrees with computations made for the case $p = 1$.
\end{remark}

\begin{remark}
For general foliations, the antibasic cohomology groups can be infinite-dimensional.
See Example~\ref{nonRiemFoliationFirstExample}.
\end{remark}

The following corollary follows in the standard way.

\begin{corollary}
\label{HodgeOrthogDecompCorollary}We have the following $L^{2}$-orthogonal
decomposition:%
\begin{equation*}
\Omega _{a}^{k}=\mathcal{H}_{a}^{k}\oplus \func{im}\left( \left. \delta
\right\vert _{\Omega _{a}^{k+1}}\right) \oplus \func{im}\left( \left.
d_{a}\right\vert _{\Omega _{a}^{k-1}}\right) .
\end{equation*}

\begin{proof}
We utilize the spectral theorem again, noting the eigenform decomposition $%
\Delta _{a}\geq 0$. For any smooth antibasic $k$-form $\alpha ,$ $\Delta
_{a}\alpha =\left( d_{a}+\delta _{a}\right) ^{2}\alpha =0$ if and only if 
\begin{eqnarray*}
0 &=&\left\langle \left( d_{a}+\delta _{a}\right) \alpha ,\left(
d_{a}+\delta _{a}\right) \alpha \right\rangle \\
&=&\left\langle d_{a}\alpha ,d_{a}\alpha \right\rangle +\left\langle \delta
\alpha ,\delta \alpha \right\rangle ,
\end{eqnarray*}%
if and only if $d_{a}\alpha =0$ and $\delta \alpha =0$. Because $d_{a}$ and $%
\delta _{a}$ commute with $\Delta _{a}$, the spaces of forms $\Omega
_{a}^{k,+}$ with positive $\Delta _{a}$ eigenvalues are mapped
isomorphically by $d_{a}+\delta $. Thus%
\begin{equation*}
\Omega _{a}^{k,+}\cong d_{a}\left( \Omega _{a}^{k,+}\right) \oplus \delta
\left( \Omega _{a}^{k,+}\right) \subseteq \Omega _{a}^{k+1,+}\oplus \Omega
_{a}^{k-1,+}.
\end{equation*}%
Then 
\begin{equation*}
\Omega _{a}^{\ast }=\mathcal{H}_{a}^{\ast }+\func{im}\left( \left. \delta
\right\vert _{\Omega _{a}^{\ast }}\right) +\func{im}\left( \left.
d_{a}\right\vert _{\Omega _{a}^{\ast }}\right) .
\end{equation*}%
Also $\func{im}\left( \left. \delta \right\vert _{\Omega _{a}^{k+1}}\right) $
and $\func{im}\left( \left. d\right\vert _{\Omega _{a}^{k-1}}\right) $ are
orthogonal: 
\begin{equation*}
\left\langle d_{a}\alpha ,\delta \beta \right\rangle =\left\langle
d_{a}^{2}\alpha ,\beta \right\rangle =0
\end{equation*}%
for all $\alpha $, $\beta $ and likewise if $\gamma \in \mathcal{H}_{a}^{k}$%
, then for all antibasic forms $\eta ,\theta $ we have%
\begin{eqnarray*}
\left\langle d_{a}\eta ,\gamma \right\rangle &=&\left\langle \eta ,\delta
\gamma \right\rangle =0, \\
\left\langle \delta \theta ,\gamma \right\rangle &=&\left\langle \theta
,d_{a}\gamma \right\rangle =0.
\end{eqnarray*}%
Therefore, the result follows.
\end{proof}
\end{corollary}

\begin{remark}
For the same reason that $\delta $ can be used in place of $d$ in computing
de Rham cohomology, the same reasoning shows from the Hodge theorem (in the
Riemannian foliation case) that 
\begin{equation*}
H_{a}^{k}\left( M,\mathcal{F}\right) \cong \frac{\ker \left(d_{a}:\Omega
_{a}^{k}\rightarrow \Omega _{a}^{k+1}\right)}{\func{im}\left(d_{a}:\Omega
_{a}^{k-1}\rightarrow \Omega _{a}^{k}\right)}.
\end{equation*}
\end{remark}

We now use the above formula for antibasic cohomology to prove the foliated
homotopy invariance of antibasic cohomology in the case of Riemannian
foliations.

\begin{lemma}
\label{pullbackLinearMapLemma}Let $\left( M,\mathcal{F}\right) $ and $\left(
M^{\prime },\mathcal{F}^{\prime }\right) $ be Riemannian foliations of
closed manifolds with bundle-like metrics, and let $f$ be a foliated map
from $\left( M,\mathcal{F}\right) \rightarrow \left( M^{\prime },\mathcal{F}%
^{\prime }\right) $. Then $P_{a}f^{\ast }P_{a}^{\prime }$ induces a linear
map on antibasic cohomology.

\begin{proof}
We consider antibasic cohomology through the isomorphism $H_{a}^{k}\left( M,%
\mathcal{F}\right) \cong \frac{\ker \left(d_{a}:\Omega _{a}^{k}\rightarrow \Omega
_{a}^{k+1}\right)}{\func{im}\left(d_{a}:\Omega _{a}^{k-1}\rightarrow \Omega _{a}^{k}\right)}$.
Then we have the following equation on antibasic forms, using (\ref%
{anti_d_ProjectionCommutator}) and the fact that pullbacks by foliated maps
preserve the basic forms: 
\begin{eqnarray*}
P_{a}f^{\ast }P_{a}^{\prime }d_{a}^{\prime } &=&P_{a}f^{\ast }P_{a}^{\prime
}d \\
&=&P_{a}f^{\ast }\left( dP_{a}^{\prime }+P_{b}^{\prime }\varepsilon ^{\ast
\prime }\right) \\
&=&P_{a}f^{\ast }dP_{a}^{\prime }+P_{a}f^{\ast }P_{b}^{\prime }\varepsilon
^{\ast \prime } \\
&=&P_{a}df^{\ast }P_{a}^{\prime }+P_{a}P_{b}f^{\ast }P_{b}^{\prime
}\varepsilon ^{\ast \prime } \\
&=&P_{a}dP_{a}f^{\ast }P_{a}^{\prime }=d_{a}P_{a}f^{\ast }P_{a}^{\prime }.
\end{eqnarray*}%
Therefore, $P_{a}f^{\ast }P_{a}^{\prime }$ induces a linear map from $%
d_{a}^{\prime }$-cohomology to $d_{a}$-cohomology.
\end{proof}
\end{lemma}

\begin{theorem}
(Foliated Homotopy Axiom of Antibasic Cohomology)\label{homotopyAxiomTheorem}
Let $\left( M,\mathcal{F}\right) $ and $\left( M^{\prime },\mathcal{F}%
^{\prime }\right) $ be Riemannian foliations of closed manifolds with
bundle-like metrics, and let $f_1$ and $f_2$ be two foliated maps from $\left( M,%
\mathcal{F}\right) \rightarrow \left( M^{\prime },\mathcal{F}^{\prime
}\right) $. If $f_1$ is foliated homotopic to $f_2$, then $P_{a}f_1^{\ast
}P_{a}^{\prime }$ and $P_{a}f_2^{\ast }P_{a}^{\prime }$ induce the same map on
antibasic cohomology.

\begin{proof}
Again we view antibasic cohomology through the isomorphism $H_{a}^{k}\left(
M,\mathcal{F}\right) \cong \frac{\ker\left( d_{a}:\Omega _{a}^{k}\rightarrow
\Omega _{a}^{k+1}\right)}{\func{im}\left(d_{a}:\Omega _{a}^{k-1}\rightarrow \Omega
_{a}^{k}\right)}$. Let $H:[0,1]\times M\rightarrow M^{\prime }$ be a foliated
homotopy such that $H\left( 0,x\right) =f_1\left( x\right) $ and $H\left(
1,x\right) =f_2(x)$. Then, define $j_{s}:M\rightarrow \lbrack 0,1]\times M$ by 
$j_{s}\left( x\right) =(s,x)$, and let $h:\Omega ^{k}\left( M^{\prime
}\right) \rightarrow \Omega ^{k-1}\left( M\right) $ by%
\begin{equation*}
h\left( \sigma \right) =\int_{0}^{1}j_{s}^{\ast }\left( \partial
_{t}\lrcorner H^{\ast }\sigma \right) ~ds.
\end{equation*}%
Note that $h$ preserves the basic forms since $H$ is a foliated homotopy. By
a standard calculation, we have that%
\begin{equation*}
f_2^{\ast }-f_1^{\ast }=dh+hd
\end{equation*}%
on $\Omega \left( M^{\prime }\right) $. Then we apply $P_{a}$ on the left
and $P_{a}^{\prime }$ on the right and use the equation $P_{a}dP_{a}=P_{a}d$
to get%
\begin{eqnarray*}
P_{a}f_2^{\ast }P_{a}^{\prime }-P_{a}f_1^{\ast }P_{a}^{\prime }
&=&P_{a}dhP_{a}^{\prime }+P_{a}hdP_{a}^{\prime } \\
&=&P_{a}dP_{a}P_{a}hP_{a}^{\prime }+P_{a}h\left( P_{a}^{\prime
}d-P_{b}^{\prime }\varepsilon ^{\ast \prime }\right) \\
&=&P_{a}dP_{a}P_{a}hP_{a}^{\prime }+P_{a}hP_{a}^{\prime }P_{a}^{\prime
}d-P_{a}hP_{b}^{\prime }\varepsilon ^{\ast \prime } \\
&=&d_{a}\left( P_{a}hP_{a}^{\prime }\right) +\left( P_{a}hP_{a}^{\prime
}\right) d_{a}^{\prime }-P_{a}P_{b}hP_{b}^{\prime }\varepsilon ^{\ast \prime
} \\
&=&d_{a}\left( P_{a}hP_{a}^{\prime }\right) +\left( P_{a}hP_{a}^{\prime
}\right) d_{a}^{\prime }.
\end{eqnarray*}%
Thus $P_{a}hP_{a}^{\prime }$ is a chain homotopy between $P_{a}f_2^{\ast
}P_{a}^{\prime }$ and $P_{a}f_1^{\ast }P_{a}^{\prime }$.
\end{proof}
\end{theorem}

\begin{corollary}
(Foliated Homotopy Invariance of Antibasic Cohomology) \label%
{homotopyInvarianceCorollary}If $\left( M,\mathcal{F}\right) $ and $\left(
M^{\prime },\mathcal{F}^{\prime }\right) $ are Riemannian foliations of
closed manifolds with bundle-like metrics that foliated homotopy equivalent,
then their antibasic cohomology groups are isomorphic.

\begin{proof}
Suppose that $f_1:M\rightarrow M^{\prime }$ is a foliated map such that there
exists a foliated map $f_2:M^{\prime }\rightarrow M$ such that $f_1\circ f_2$ and $%
f_2\circ f_1$ are each foliated homotopic to the identity. Then we have that%
\begin{eqnarray*}
P_{a}^{\prime }\left( f_1\circ f_2\right) ^{\ast }P_{a}^{\prime } &=&\operatorname{Id}^{\prime
}:H_{a}^{k}\left( M^{\prime }\right) \rightarrow H_{a}^{k}\left( M^{\prime
}\right) , \\
P_{a}\left( f_2\circ f_1\right) ^{\ast }P_{a} &=&\operatorname{Id}:H_{a}^{k}\left( M\right)
\rightarrow H_{a}^{k}\left( M\right) .
\end{eqnarray*}%
Since the pullback by $f_2$ preserves the basic forms,%
\begin{eqnarray*}
\operatorname{Id}^{\prime } &=&P_{a}^{\prime }f_2^{\ast }f_1^{\ast }P_{a}^{\prime
}=P_{a}^{\prime }f_2^{\ast }\left( P_{a}+P_{b}\right) f_1^{\ast }P_{a}^{\prime }
\\
&=&P_{a}^{\prime }f_2^{\ast }P_{a}P_{a}f_1^{\ast }P_{a}^{\prime }+P_{a}^{\prime
}f_2^{\ast }P_{b}f_1^{\ast }P_{a}^{\prime } \\
&=&P_{a}^{\prime }f_2^{\ast }P_{a}P_{a}f_1^{\ast }P_{a}^{\prime }+P_{a}^{\prime
}P_{b}^{\prime }f_2^{\ast }P_{b}f_1^{\ast }P_{a}^{\prime }=\left( P_{a}^{\prime
}f_2^{\ast }P_{a}\right) \left( P_{a}f_1^{\ast }P_{a}^{\prime }\right) ,
\end{eqnarray*}%
and similarly $\left( P_{a}f_1^{\ast }P_{a}^{\prime }\right) \left(
P_{a}^{\prime }f_2^{\ast }P_{a}\right) =\operatorname{Id}$, so we must have that $%
P_{a}f_1^{\ast }P_{a}^{\prime }$ is an isomorphism from $H_{a}^{k}\left(
M^{\prime }\right) $ to $H_{a}^{k}\left( M\right) $.
\end{proof}
\end{corollary}

\section{Properties and applications\label{propertiesApplicationsSection}}

First we consider the simple case when the operators $\varepsilon $ and $%
\varepsilon ^{\ast }$ are zero. In this case, the antibasic Betti numbers
can be computed from the ordinary Betti numbers and basic Betti numbers.

\begin{proposition}
\label{MeanBasicInvNormalBdleProp}Suppose that $\left( M,\mathcal{F}\right) $
is a Riemannian foliation of a closed manifold of dimension $n$, such that
the normal bundle $(T\mathcal{F})^\perp$ is involutive. Then for any bundle-like
metric, the antibasic cohomology and basic cohomology add to the ordinary
cohomology. That is, for $0\leq k\leq n$ 
\begin{equation*}
H^{k}\left( M\right) \cong H_{b}^{k}\left( M,\mathcal{F}\right) \oplus
H_{a}^{k}\left( M,\mathcal{F}\right) .
\end{equation*}

\begin{proof}
First, we choose a bundle-like metric so that the mean curvature is basic;
this can always be done \cite{Dom}. The operator $\varepsilon ^{\ast }$
satisfies%
\begin{eqnarray*}
\varepsilon ^{\ast } &=&-\kappa _{a}\wedge +\left( -1\right) ^{k}\left( \chi
_{\mathcal{F}}\lrcorner \right) \left( \varphi _{0}\wedge \right) \\
&=&0
\end{eqnarray*}%
under the hypotheses, since $\varphi _{0}=0$ if and only if the normal
bundle is involutive. Then, by Theorem \ref{antibasicLaplacianTheorem}, the
antibasic Laplacian is precisely a restriction of the ordinary Laplacian.
Similarly, by formula (\ref{deltaBasicProjForm}) and the results of \cite%
{PaRi}, the basic Laplacian is a restriction of the ordinary Laplacian. Thus
the ordinary Laplacian preserves the basic and antibasic forms, and it
decomposes as a direct sum of the basic and antibasic Laplacians. The
harmonic forms decompose into basic and antibasic parts, and the Hodge
theorem implies the result.
\end{proof}
\end{proposition}

\begin{corollary}
Suppose that $\left( M,\mathcal{F}\right) $ is a Riemannian foliation of a
closed manifold of dimension $n$, such that the normal bundle $(T%
\mathcal{F})^\perp$ is involutive. Then $\dim H_{b}^{k}\left( M,\mathcal{F}\right)
\leq \dim H^{k}\left( M\right) $, $\dim H_{a}^{k}\left( M,\mathcal{F}\right)
\leq \dim H^{k}\left( M\right) .$
\end{corollary}

\begin{remark}
It was essentially already known that $\dim H_{b}^{k}\left( M,\mathcal{F}%
\right) \leq \dim H^{k}\left( M\right) $ in this case, because using \cite%
{PaRi} and \cite{Dom} we see that for a metric with basic mean curvature, $%
\delta _{b}=\delta $ when restricted to basic forms.
\end{remark}

\begin{proposition}
Suppose that $\left( M,\mathcal{F},g\right) $ is a Riemannian foliation of a
closed manifold of dimension $n$ with bundle-like metric, such that the mean
curvature form is basic and the normal bundle $N\mathcal{F}=\left( T\mathcal{%
F}\right) ^{\bot }$ is involutive. Then the wedge product induces a bilinear
product on basic and antibasic cohomology: 
\begin{equation*}
\wedge :H_{b}^{r}\left( M,\mathcal{F}\right) \otimes H_{a}^{s}\left( M,%
\mathcal{F}\right) \rightarrow H_{a}^{r+s}\left( M,\mathcal{F}\right) .
\end{equation*}

\begin{proof}
The operator $\varepsilon =0$ under the assumptions, so that both $d$ and $%
\delta $ restrict to both basic and antibasic forms. The wedge product of a
basic and antibasic form is antibasic (since $P_{b}\left( P_{b}\alpha \wedge
\beta \right) =P_{b}\alpha \wedge P_{b}\beta $ from \cite{PaRi}), so that
the result follows from the standard result in de Rham cohomology.
\end{proof}
\end{proposition}

\vspace{1pt}In more generality, if the mean curvature is basic but without
the assumption on $\varphi _{0}$, the same result is true for $k=0$.

\begin{proposition}
\label{H0_Prop}Suppose that $\left( M,\mathcal{F},g\right) $ is a Riemannian
foliation of a closed manifold of dimension $n$. Then 
\begin{equation*}
H^{0}\left( M\right) \cong H_{b}^{0}\left( M,\mathcal{F}\right) \oplus
H_{a}^{0}\left( M,\mathcal{F}\right) .
\end{equation*}%
In particular, if $M$ is connected, then $H_{b}^{0}\left( M,\mathcal{F}%
\right) \cong \mathbb{R}$ and $H_{a}^{0}\left( M,\mathcal{F}\right) \cong
\left\{ 0\right\} $.

\begin{proof}
We first choose a bundle-like metric such that the mean curvature is basic.
By Theorem \ref{antibasicLaplacianTheorem}, $\Delta _{a}$ is the restriction
of $\Delta +\delta P_{b}\varepsilon ^{\ast }+P_{b}\varepsilon ^{\ast }\delta 
$ to $\Omega _{a}^{\ast }\left( M\right) $, and on functions
this is $\Delta +\delta P_{b}\varepsilon ^{\ast }$. But also $\varepsilon
^{\ast }=-\kappa _{a}\wedge +\left( -1\right) ^{k}\left( \chi _{\mathcal{F}%
}\lrcorner \right) \left( \varphi _{0}\wedge \right) =0$ on functions, so
that $\Delta _{a}$ is the restriction of the ordinary Laplacian. Also, by
the results of \cite{PaRi}, $\Delta _{b}$ is the restriction of $\Delta
+\varepsilon d+d\varepsilon $ to $\Omega _{b}^{\ast }\left( M%
\right) $, and on functions this is $\Delta +\varepsilon d$, but in our case 
$\varepsilon =\left( -1\right) ^{k}\left( \varphi _{0}\lrcorner \right)
\left( \chi _{\mathcal{F}}\wedge \right) $ is zero on basic one-forms so
that $\Delta _{b}$ is the restriction of $\Delta $. Thus $\Delta $ is the
orthogonal direct sum of the restrictions to basic and antibasic functions,
and the result follows from the Hodge theorem.
\end{proof}
\end{proposition}

\begin{proposition}
\label{degree 1 cohomology Prop}Suppose that $\left( M,\mathcal{F}\right) $
is a Riemannian foliation on a closed, connected manifold. Then,%
\begin{equation*}
\dim H^{1}\left( M\right) \leq \dim H_{b}^{1}\left( M,\mathcal{F}\right)
+\dim H_{a}^{1}\left( M,\mathcal{F}\right) .
\end{equation*}

\begin{proof}
First, we choose a bundle-like metric with basic mean curvature. Given a $%
\Delta $-harmonic form $\beta $, consider $P_{a}\beta $. We see that 
\begin{equation*}
d_a(P_{a}\beta)=P_{a}d\left( P_{a}\beta \right) =P_{a}\left( P_{a}\left( d\beta \right)
-P_{b}\left( \varepsilon ^{\ast }\beta \right) \right) =0.
\end{equation*}%
Also, 
\begin{equation*}
\delta _{a}P_{a}\beta =\delta P_{a}\beta =P_{a}\delta \beta +\varepsilon
\left( P_{b}\beta \right) =0,
\end{equation*}%
because $\varepsilon =0$ on basic one-forms. Thus the map $\beta \mapsto
P_{a}\beta $ maps harmonic one-forms to antibasic harmonic one-forms. The
kernel of this map is the set of basic forms $\beta$ such that $d\beta =0$ and $%
0=\delta \beta =\left( \delta _{b}-\varepsilon \right) \beta =\delta _{b}\beta 
$, since $\varepsilon $ is zero on basic one-forms. Thus the kernel is the
set of $\Delta _{b}$-harmonic forms. By the Hodge theorem, the result
follows, since $H^{1}\left( M\right) \cong H_{b}^{1}\left( M,\mathcal{F}%
\right) \oplus P_{a}\left( \mathcal{H}^{1}\left( M\right) \right) \subseteq
H_{b}^{1}\left( M,\mathcal{F}\right) \oplus \mathcal{H}_{a}^{1}\left( M,\mathcal{F}%
\right) $.
\end{proof}
\end{proposition}

\begin{remark}
Note that in general $H_{b}^{1}\left( M\right) \hookrightarrow H^{1}\left(
M\right) $ is an injection for all foliations, so always \linebreak$\dim
H_{b}^{1}\left( M\right) \leq \dim H^{1}\left( M\right) $. Thus, if $%
H_{a}^{1}\left( M,\mathcal{F}\right) \cong \left\{ 0\right\} $, then $%
H_{b}^{1}\left( M\right) \cong H^{1}\left( M\right) $, so that every
harmonic one-form is basic.
\end{remark}

Another simple class of examples of Riemannian foliations occurs when the
orbits of a compact connected Lie group action all have the same dimension.
In this case, we may choose a metric such that the Lie group acts by
isometries. The Lie group acts on differential forms by pullback, and this
action commutes with $d$ and $\delta $. Thus, if we decompose the
differential forms according to the irreducible representations $\rho \in 
\widehat{G}$ of $G$, we have the $L^{2}$-orthogonal direct sum%
\begin{equation*}
\Omega ^{\ast }\left( M\right) =\displaystyle\bigoplus\limits_{\rho \in 
\widehat{G}}\Omega ^{\ast ,\rho }\left( M\right)
\end{equation*}%
where $\Omega ^{\ast ,\rho }\left( M\right) $ is the space of differential
forms of type $\rho :G\rightarrow U\left( V_{\rho }\right) $. That is,%
\begin{equation*}
\Omega ^{\ast ,\rho }\left( M\right) =\displaystyle\bigcup\limits_{f\in 
\mathrm{Hom}_{G}\left( V_{\rho },\Omega ^{\ast }\left( M\right) \right)
}f\left( V_{\rho }\right) .
\end{equation*}%
Because of the metric invariance, both $d$ and $\delta $ respect this
decomposition. It is well-known that the harmonic forms are always invariant
(i.e. belong to $\Omega ^{\ast ,\rho _{0}}\left( M\right) $, where $\rho
_{0} $ is the trivial representation). Also, for the foliation $\mathcal{F}$
by $G $-orbits, we have $\Omega _{b}^{\ast }\left( M\right)
\subseteq \Omega ^{\ast ,\rho _{0}}\left( M\right) $. We let $d_{j},\delta
_{j}$ refer to the restrictions of $d,\delta $ to $\Omega ^{j}$, and we let $%
d_{a,j},\delta _{aj},d_{b,j},\delta _{bj}$ denote the corresponding
restrictions to basic and antibasic forms. We use the superscript $\rho $ to
denote further restrictions to $\Omega ^{\ast ,\rho }\left( M\right) $. Then
we have 
\begin{eqnarray*}
d_{j} &=&\displaystyle\bigoplus\limits_{\rho \in \widehat{G}}d_{j}^{\rho }, \\
d_{bj} &=&d_{bj}^{\rho _{0}} ,\\
d_{aj} &=&d_{aj}^{\rho _{0}}\oplus \displaystyle\bigoplus\limits_{\substack{ %
\rho \in \widehat{G}  \\ \rho \neq \rho _{0}}}d_{j}^{\rho }, \\
\delta _{bj} &=&\delta _{bj}^{\rho _{0}} ,\\
\delta _{aj} &=&\delta _{aj}^{\rho _{0}}\oplus \displaystyle\bigoplus\limits 
_{\substack{ \rho \in \widehat{G}  \\ \rho \neq \rho _{0}}}\delta _{j}^{\rho
}.
\end{eqnarray*}%
Thus, in computing either the basic or antibasic cohomology, it is
sufficient to restrict to invariant forms. The result below follows.

\begin{proposition}
\label{groupActionProp}Let $G$ be a connected, compact Lie group that acts
on a connected closed manifold $M$ by isometries. Let $\left( M,\mathcal{F}%
,g\right) $ be the Riemannian foliation with bundle-like metric given by the 
$G$-orbits. Then%
\begin{equation*}
H_{b}^{j}\left( M,\mathcal{F}\right) \cong \frac{\ker d_{bj}^{\rho _{0}}}{%
\func{im}d_{b\left( j-1\right) }^{\rho _{0}}};~~H_{a}^{j}\left( M,\mathcal{F}%
\right) \cong \frac{\ker \delta _{aj}^{\rho _{0}}}{\func{im}\delta _{a\left(
j+1\right) }^{\rho _{0}}}.
\end{equation*}%
In particular,%
\begin{equation*}
H_{b}^{0}\left( M,\mathcal{F}\right) \cong \mathbb{R;~~}H_{a}^{0}\left( M,%
\mathcal{F}\right) \cong \left\{ 0\right\} .
\end{equation*}

\begin{proof}
The first part follows from the discussion above. Next, observe that all $G$%
-invariant functions are basic, so that $\Omega _{b}^{0}\left( M,\mathcal{F}%
\right) =\Omega ^{0,\rho _{0}}\left( M\right) $, so that $\ker \left(
d_{b0}^{\rho _{0}}\right) =\ker \left( d_{0}^{\rho _{0}}\right) $ consists
of the constant functions, and $\ker \delta _{a0}^{\rho _{0}}=\left\{
0\right\} $. The second part also follows from Proposition \ref{H0_Prop},
since the mean curvature is always basic in this case.
\end{proof}
\end{proposition}

\section{The case of Riemannian flows\label{RiemannianFlowSection}}

In this section, we study tautness and cohomology for Riemannian flows. The
following result shows a relationship between basic and antibasic cohomology
when the flow is taut. We will see evidence of this behavior in Example \ref%
{HopfFibExample} and Example \ref{CarExample}. We will use standard
techniques in the study of these flows, which can also be found for example
in \cite{Car}, \cite{Mo}, \cite{HabEMT}.

\begin{proposition}
\label{Prop_r2r+1}Suppose that $\left( M,\mathcal{F}\right) $ is a
Riemannian flow on a closed manifold with bundle-like metric $g$ with basic
mean curvature $\kappa $ and characteristic form $\chi_{\mathcal F} $. If $\kappa =0$,
then for all $r$, the map $\alpha \mapsto \chi_{\mathcal F} \wedge \alpha $ maps basic
harmonic $r$-forms to antibasic harmonic $\left( r+1\right) $-forms
injectively. Thus, $\dim \left( H_{a}^{r+1}\left( M,\mathcal{F}\right)
\right) \geq \dim \left( H_{b}^{r}\left( M,\mathcal{F}\right) \right) $
whenever $\left[ \kappa \right] =0\in H_{b}^{1}\left( M,\mathcal{F}\right) $.

\begin{proof}
Suppose that $\left[ \kappa \right] =0$. We then choose a bundle-like metric
such that $\kappa =0$. Note that $\varphi_0$ (from Rummler's formula) is basic for 
every Riemannian flow with bundle-like metric
with basic mean curvature, and so it is certainly basic in this case.
With this metric, for any basic harmonic $r$-form $%
\alpha $,%
\begin{eqnarray*}
d\left( \chi_{\mathcal F} \wedge \alpha \right) &=&d\chi_{\mathcal F} \wedge \alpha -\chi_{\mathcal F} \wedge
d\alpha \\
&=&\varphi _{0}\wedge \alpha ,
\end{eqnarray*}%
which is basic, 
so that $d_{a}\left( \chi_{\mathcal F} \wedge \alpha \right) =0$. Let $%
\chi_{\mathcal F} ^{\#}=\xi $, and choose the usual adapted orthonormal frame $\left\{
b_{i}\right\} =\left\{ e_{i}\right\} \cup \left\{ \xi \right\} $ near a
point, where the $e_{i}$ are basic and $\nabla ^{Q}$-parallel at the point
in question. Note that $\delta \chi_{\mathcal F} =0$ because the metric is bundle-like,
and then%
\begin{eqnarray*}
\delta \left( \chi_{\mathcal F} \wedge \alpha \right) &=&-\sum_{i}b_{i}\lrcorner \nabla
_{b_{i}}^{M}\left( \chi_{\mathcal F} \wedge \alpha \right) \\
&=&-\sum_{i}b_{i}\lrcorner \left( \nabla _{b_{i}}^{M}\chi_{\mathcal F} \wedge \alpha
+\chi_{\mathcal F} \wedge \nabla _{b_{i}}^{M}\alpha \right) \\
&=&\left( \delta \chi_{\mathcal F} \right) \alpha +\sum_{i}\nabla _{b_{i}}^{M}\chi_{\mathcal F} \wedge
\left( b_{i}\lrcorner \alpha \right) -\nabla _{\xi }^{M}\alpha +\chi_{\mathcal F} \wedge
\sum_{i}\left( b_{i}\lrcorner \nabla _{b_{i}}^{M}\alpha \right) \\
&=&\sum_{i}\nabla _{b_{i}}^{M}\chi_{\mathcal F} \wedge \left( b_{i}\lrcorner \alpha
\right) -\nabla _{\xi }^{M}\alpha -\chi_{\mathcal F} \wedge \delta \alpha .
\end{eqnarray*}%
But since $\alpha $ is basic harmonic, $\delta \alpha =\delta _{b}\alpha
-\varepsilon \alpha =0+\varphi _{0}\lrcorner \left( \chi_{\mathcal F} \wedge \alpha
\right) =\chi_{\mathcal F} \wedge \left( \varphi _{0}\lrcorner \alpha \right) .$ Thus, $%
\chi_{\mathcal F} \wedge \delta \alpha =0$, so that%
\begin{eqnarray*}
\delta \left( \chi_{\mathcal F} \wedge \alpha \right) &=&\sum_{i}\nabla _{b_{i}}^{M}\chi_{\mathcal F}
\wedge \left( b_{i}\lrcorner \alpha \right) -\nabla _{\xi }^{M}\alpha \\
&=&\sum_{i}\nabla _{e_{i}}^{M}\chi_{\mathcal F} \wedge \left( e_{i}\lrcorner \alpha
\right) -\nabla _{\xi }^{M}\alpha \\ 
&=&\sum_{i}\left( he_{i}\right) ^{\flat }\wedge \left( e_{i}\lrcorner \alpha
\right) -\nabla _{\xi }^{M}\alpha \\
&=&\sum_{i,j}g(he_{i},e_{j}) e^{j}\wedge \left( e_{i}\lrcorner
\alpha \right) -\nabla _{\xi }^{M}\alpha \\
&=&-\sum_{j}e^{j}\wedge \left( \left( he_{j}\right) \lrcorner \alpha \right)
-\nabla _{\xi }^{M}\alpha ,
\end{eqnarray*}%
where the skew-adjoint O'Neill tensor $h$ satisfies $hX=\nabla _{X}^{M}\xi $
for $X\in \Gamma \left( N\mathcal{F}\right) $. Now observe from one hand
that $\left( \nabla _{\xi }^{M}\alpha \right) \left( \xi
,e_{i_{1}},...,e_{i_{r-1}}\right) =0$ since $\alpha $ is basic and $\kappa=0$. 
On the other
hand, we use%
\begin{equation*}
\nabla _{\xi }^{M}Z=\nabla _{\xi }^{Q}Z+h\left( Z\right) -\kappa \left(
Z\right) \xi =\nabla _{\xi }^{Q}Z+h\left( Z\right)
\end{equation*}%
for $Z\in \Gamma \left( N\mathcal{F}\right) $ to compute 
\begin{eqnarray*}
\left( \nabla _{\xi }^{M}\alpha \right) \left(
e_{i_{1}},...,e_{i_{r}}\right) &=&\xi \left( \alpha \left(
e_{i_{1}},...,e_{i_{r}}\right) \right) -\sum_{k}\alpha \left(
e_{i_{1}},...,\nabla _{\xi }^{M}e_{i_{k}},...,e_{i_{r}}\right) \\
&=&\xi \left( \alpha \left( e_{i_{1}},...,e_{i_{r}}\right) \right)
-\sum_{k}\alpha \left( e_{i_{1}},...,he_{i_{k}},...,e_{i_{r}}\right) \\
&=&\left( \nabla _{\xi }^{Q}\alpha \right) \left(
e_{i_{1}},...,e_{i_{r}}\right) -\sum_{k}\alpha \left(
e_{i_{1}},...,he_{i_{k}},...,e_{i_{r}}\right) \\
&=&-\sum_{k}\alpha \left( e_{i_{1}},...,he_{i_{k}},...,e_{i_{r}}\right) .
\end{eqnarray*}%
Now we write%
\begin{eqnarray*}
\nabla _{\xi }^{M}\alpha &=&\frac{1}{r!}\sum_{\left( i_{1},...,i_{r}\right)
}\left( \nabla _{\xi }^{M}\alpha \right) \left(
e_{i_{1}},...,e_{i_{r}}\right) e^{i_{1}}\wedge ...\wedge e^{i_{r}} \\
&=&-\frac{1}{r!}\sum \alpha \left(
e_{i_{1}},...,he_{i_{k}},...,e_{i_{r}}\right) e^{i_{1}}\wedge ...\wedge
e^{i_{r}} \\
&=&-\frac{1}{r!}\sum \left( -1\right) ^{k-1}\left( he_{i_{k}}\lrcorner
\alpha \right) \left( e_{i_{1}},...,\widehat{e_{i_{k}}},...,e_{i_{r}}\right)
e^{i_{1}}\wedge ...\wedge e^{i_{r}} \\
&=&-\frac{r}{r!}\sum \left( he_{\ell }\lrcorner \alpha \right) \left(
e_{i_{1}},...,e_{i_{r-1}}\right) e^{\ell }\wedge e^{i_{1}}\wedge ...\wedge
e^{i_{r-1}} \\
&=&-\sum_{\ell }e^{\ell }\wedge \left( he_{\ell }\lrcorner \alpha \right) .
\end{eqnarray*}%
Thus, substituting we have%
\begin{eqnarray*}
\delta \left( \chi_{\mathcal F} \wedge \alpha \right) &=&-\sum_{j}e^{j}\wedge \left(
\left( he_{j}\right) \lrcorner \alpha \right) -\nabla _{\xi }^{M}\alpha \\
&=&-\sum_{j}e^{j}\wedge \left( \left( he_{j}\right) \lrcorner \alpha \right)
+\sum_{\ell }e^{\ell }\wedge \left( he_{\ell }\right) \lrcorner \alpha \\
&=&0,
\end{eqnarray*}%
so that $\chi_{\mathcal F} \wedge \alpha \in \mathcal{H}_{a}^{r+1}\left( M,\mathcal{F}%
\right) $. If $\alpha $ is nonzero, then $\chi_{\mathcal F} \wedge \alpha $ is nonzero,
so the class $\left[ \chi_{\mathcal F} \wedge \alpha \right] $ is nontrivial.
\end{proof}
\end{proposition}

\begin{remark} \label{rem:h1a}
In particular, if $\left[ \kappa \right] =0\in H_{b}^{1}\left( M,\mathcal{F}%
\right) $, then $\dim H_{a}^{1}\left( M,\mathcal{F}\right) \geq 1$.
\end{remark}

\begin{lemma}
\label{Delta_a_alpha_Lemma}Suppose that $\left( M,\mathcal{F}\right) $ is a
Riemannian flow on a closed, connected manifold, with a bundle-like metric
chosen so that the mean curvature is basic. Then for any antibasic one-forms $%
\alpha $ and $\beta $,%
\begin{equation*}
\left\langle \Delta _{a}\alpha ,\beta \right\rangle =\left\langle \Delta
\alpha ,\beta \right\rangle -\int_M P_{b}\left( \chi_{\mathcal F} ,\alpha \right)
P_{b}\left( \chi_{\mathcal F} ,\beta \right) \left\vert \varphi _{0}\right\vert ^{2} dv_g.
\end{equation*}

\begin{proof}
From Theorem \ref{antibasicLaplacianTheorem}, 
\begin{equation*}
\Delta _{a}\alpha =\Delta \alpha +\delta P_{b}\varepsilon ^{\ast }\alpha
+P_{b}\varepsilon ^{\ast }\delta \alpha .
\end{equation*}%
Since $\varepsilon ^{\ast }\delta \alpha =0$ and 
\begin{equation*}
\varepsilon ^{\ast }\alpha =-\chi_{\mathcal F} \lrcorner \left( \varphi _{0}\wedge \alpha
\right) =-\left(\chi_{\mathcal F},\alpha\right) \varphi _{0},
\end{equation*}%
we have%
\begin{eqnarray*}
\delta P_{b}\varepsilon ^{\ast }\alpha  &=&-\delta P_{b}\left(\left(\chi_{\mathcal F},\alpha\right) \varphi _{0}\right) =-\delta \left( P_{b} \left(\chi_{\mathcal F},\alpha\right)  \varphi _{0}\right)  \\
&=&df\lrcorner \varphi _{0}-f\delta \varphi _{0},
\end{eqnarray*}%
where $f=P_{b}\left( \chi_{\mathcal F} ,\alpha \right)$. Now, 
using $\delta \varphi_0=\delta_b\varphi_0-\varepsilon\varphi_0$ with $\varepsilon\varphi_0=-\varphi_0\lrcorner(\chi_{\mathcal F}\wedge\varphi_0)$, we write
\begin{eqnarray*}
\left\langle \Delta _{a}\alpha ,\beta \right\rangle  &=&\left\langle \Delta
\alpha ,\beta \right\rangle +\left\langle df\lrcorner \varphi _{0}-f\delta
\varphi _{0},\beta \right\rangle  \\
&=&\left\langle \Delta \alpha ,\beta \right\rangle +\left\langle \varphi
_{0},df\wedge \beta \right\rangle -\left\langle f\delta \varphi _{0},\beta
\right\rangle  \\
&=&\left\langle \Delta \alpha ,\beta \right\rangle -\left\langle f\left(
\delta _{b}-\varepsilon \right) \varphi _{0},\beta \right\rangle  \\
&=&\left\langle \Delta \alpha ,\beta \right\rangle +\left\langle f\varepsilon\varphi _{0},\beta \right\rangle  \\
&=&\left\langle \Delta \alpha ,\beta \right\rangle -\left\langle f\left(
\varphi _{0}\lrcorner \chi_{\mathcal F} \wedge \varphi _{0}\right) ,\beta \right\rangle 
\\
&=&\left\langle \Delta \alpha ,\beta \right\rangle -\left\langle f\left\vert
\varphi _{0}\right\vert ^{2}\chi_{\mathcal F} ,\beta \right\rangle  \\
&=&\left\langle \Delta \alpha ,\beta \right\rangle -\int_M f\left\vert \varphi
_{0}\right\vert ^{2}P_{b}\left( \chi_{\mathcal F} ,\beta \right)dv_g  \\
&=&\left\langle \Delta \alpha ,\beta \right\rangle -\int_M P_{b}\left( \chi_{\mathcal F}
,\alpha \right) P_{b}\left( \chi_{\mathcal F} ,\beta \right) \left\vert \varphi
_{0}\right\vert ^{2} dv_g.
\end{eqnarray*}%
This completes the proof. 
\end{proof}
\end{lemma}

\begin{theorem}
\label{H1zeroKbasicProp}Suppose that $\left( M,\mathcal{F}\right) $ is a
Riemannian flow on a closed, connected manifold. If $H^{1}\left( M\right)
=\left\{ 0\right\} $, then%
\begin{equation*}
\dim H_{a}^{1}\left( M,\mathcal{F}\right) =1.
\end{equation*}
If $\left[ \kappa \right] $
is a nonzero class in $H_{b}^{1}\left( M,\mathcal{F}\right) \subseteq
H^{1}\left( M\right) $, then 
\begin{equation*}
\dim H_{a}^{1}\left( M,\mathcal{F}\right) =0.
\end{equation*}
\end{theorem}

\begin{proof} We choose the bundle-like metric so that $\kappa $
is basic-harmonic $1$-form (as in \cite{MarchEtAl}). We write any antibasic one-form $\alpha $ as 
\begin{equation*}
\alpha =f\chi_{\mathcal F} +\beta =\left( f_{a}\chi_{\mathcal F} +\beta \right) +\left( f_{b}\chi_{\mathcal F}
\right) =\alpha _{1}+\alpha _{2},
\end{equation*}%
where $f_{a}=P_{a}f$, $f_{b}=P_{b}f$, $\alpha _{1}=f_{a}\chi_{\mathcal F} +\beta $, $%
\alpha _{2}=f_{b}\chi_{\mathcal F} $ and $\beta $ is an antibasic section
of $N^{\ast }\mathcal{F}$. The $L^{2}$ inner
product gives 
\begin{eqnarray}\label{splittingalpha}
\left\langle \Delta _{a}\alpha ,\alpha \right\rangle &=&\left\langle \Delta
_{a}\alpha _{1},\alpha _{1}\right\rangle +\left\langle \Delta _{a}\alpha
_{2},\alpha _{2}\right\rangle +2\left\langle \Delta_a \alpha _{1},\alpha
_{2}\right\rangle \nonumber\\
&=&\left\langle \Delta _{a}\alpha _{1},\alpha _{1}\right\rangle
+\left\langle \Delta _{a}\alpha _{2},\alpha _{2}\right\rangle +2\left\langle
\Delta _{a}\left( f_{a}\chi_{\mathcal F} \right) ,f_{b}\chi_{\mathcal F} \right\rangle +2\left\langle
\Delta _{a}\beta ,f_{b}\chi_{\mathcal F} \right\rangle\nonumber\\
&=&\left\langle \Delta\alpha _{1},\alpha _{1}\right\rangle
+\left\langle \Delta \left(f_b\chi_{\mathcal F}\right),f_b\chi_{\mathcal F}\right\rangle-\int_M f_b^2 \left\vert \varphi _{0}\right\vert ^{2}dv_g +2\left\langle
\Delta\left( f_{b}\chi_{\mathcal F} \right) ,f_{a}\chi_{\mathcal F} \right\rangle +2\left\langle  \Delta(f_{b}\chi_{\mathcal F}),
\beta \right\rangle.\nonumber\\
\end{eqnarray}%
In the last equality, we use the formula in Lemma \ref{Delta_a_alpha_Lemma}. In order to express each of the above inner product, we will compute $\Delta \left( f_b\chi_{\mathcal F} \right),$ for any basic function $f_b$. To simplify the notation, we will omit the subscript ``$b$" in $f_b$ in the following computations. First we have (keep in mind that $f=f_b$ is basic) 
\begin{equation*}
\delta \left( f\chi_{\mathcal F} \right) =-df\lrcorner \chi_{\mathcal F} +f\delta \chi_{\mathcal F} =0
\end{equation*}%
since $\chi_{\mathcal F} $ is divergence free. Therefore, $d(\delta(f\chi_{\mathcal F}))=0.$ 
Next, using Rummler's formula $d\chi_{\mathcal F}=-\kappa\wedge \chi_{\mathcal F}+\varphi_0,$ we write 
\begin{eqnarray}
d\left( f\chi_{\mathcal F} \right)  &=&fd\chi_{\mathcal F} +df\wedge \chi_{\mathcal F} =-f\kappa \wedge \chi_{\mathcal F}
+f\varphi _{0}+df\wedge \chi_{\mathcal F} ,  \notag \\
\delta d\left( f\chi_{\mathcal F} \right)  &=&\delta \left( -f\kappa \wedge \chi_{\mathcal F}
+f\varphi _{0}+df\wedge \chi_{\mathcal F} \right)   \notag \\
&=&df\lrcorner \left( \kappa \wedge \chi_{\mathcal F} -\varphi _{0}\right) -f\delta
\left( \kappa \wedge \chi_{\mathcal F} -\varphi _{0}\right) +\delta \left( df\wedge \chi_{\mathcal F}
\right) .  \label{formulaDeltaTf}
\end{eqnarray}%
\newline
To express the divergence terms in the above equality, we consider an orthonormal frame $\{b_i\}$ of $TM$ and we compute for any basic $1$-form $\theta,$  
\begin{eqnarray*}
\delta \left( \theta \wedge \chi_{\mathcal F} \right)  &=&-\sum_i b_{i}\lrcorner \nabla^M
_{b_{i}}\left( \theta \wedge \chi_{\mathcal F} \right)  \\
&=&-\sum_i b_{i}\lrcorner \left( \nabla^M_{b_{i}}\theta \wedge \chi_{\mathcal F} +\theta \wedge
\nabla^M _{b_{i}}\chi_{\mathcal F} \right)  \\
&=&\left( \delta \theta \right) \chi_{\mathcal F} +\nabla^M _{\xi }\theta -\nabla^M_{\theta
^{\#}}\chi_{\mathcal F} -\theta \wedge \delta \chi_{\mathcal F}  \\
&=&\left( \delta \theta \right) \chi_{\mathcal F} +\nabla^M_{\xi }\theta -\nabla^M _{\theta
^{\#}}\chi_{\mathcal F}  \\
&=&\left( \delta _{b}\theta \right) \chi_{\mathcal F} +\left[ \xi ,\theta ^{\#}\right]
^{\flat } \\
&=&\left( \delta _{b}\theta- \left(\kappa,\theta \right)\right)\chi_{\mathcal F}.
\end{eqnarray*}%
Therefore, we deduce for either $\theta=\kappa$ or $\theta=df$ that
$$ \delta \left( \kappa \wedge \chi_{\mathcal F} \right)=-\left\vert \kappa \right\vert ^{2}\chi_{\mathcal F} \quad\text{and}\quad \delta \left( df \wedge \chi_{\mathcal F} \right)=\left( \Delta_b f-\left( df,\kappa \right) \right) \chi_{\mathcal F}$$ 
since $\kappa $ is basic harmonic.
Then we substitute into (\ref{formulaDeltaTf}) to get%
\begin{eqnarray*}
\delta d\left( f\chi_{\mathcal F} \right)  &=&df\lrcorner \left( \kappa \wedge \chi_{\mathcal F}
\right) -df\lrcorner \varphi _{0}-f\delta \left( \kappa \wedge \chi_{\mathcal F} \right)
+f\delta \varphi _{0}+\delta \left( df\wedge \chi_{\mathcal F} \right)  \\
&=&\left( df,\kappa \right) \chi_{\mathcal F}-df\lrcorner
\varphi _{0}+f\left\vert \kappa \right\vert ^{2}\chi_{\mathcal F} +f\delta \varphi _{0}+\left( \Delta_b f-\left( df,\kappa \right) \right) \chi_{\mathcal F}\\
&=&-df\lrcorner
\varphi _{0}+f\left\vert \kappa \right\vert ^{2}\chi_{\mathcal F} +f\delta \varphi _{0}+\left( \Delta_b f \right) \chi_{\mathcal F}\\
&=&-df\lrcorner
\varphi _{0}+f\left\vert \kappa \right\vert ^{2}\chi_{\mathcal F} +f\left( \delta _{b}-\varepsilon \right) \varphi _{0}+\left( \Delta_b f \right) \chi_{\mathcal F} 
\end{eqnarray*}%
since $\delta P_{b}=P_{b}\delta -\varepsilon P_{b}$ from (\ref%
{deltaBasicProjForm}). As $\varepsilon \varphi _{0}=-\varphi _{0}\lrcorner
\left( \chi_{\mathcal F} \wedge \varphi _{0}\right) =-\left\vert \varphi _{0}\right\vert
^{2}\chi_{\mathcal F} $, we arrive at (replace $f$ by $f_b$) %
\begin{equation*}
\Delta \left( f_b\chi_{\mathcal F} \right) =-df_b\lrcorner
\varphi _{0}+f_b\left\vert \kappa \right\vert ^{2}\chi_{\mathcal F} +f_b\delta _{b}\varphi
_{0}+f_b\left\vert \varphi _{0}\right\vert ^{2}\chi_{\mathcal F} +\left( \Delta_b f_b\right)
\chi_{\mathcal F}. 
\end{equation*} 
%
In particular, one can easily get that 
\begin{equation} \label{deltafbchifa}
\left\langle \Delta\left( f_b\chi_{\mathcal F} \right),f_a\chi_{\mathcal F} \right\rangle=0 \quad\text{and}\quad \left\langle \Delta\left( f_b\chi_{\mathcal F} \right),\beta \right\rangle=0,
\end{equation}
since $\beta$ is antibasic and orthogonal to $\xi$. Also, we have that 
\begin{equation}\label{deltafbchifb}
\left\langle \Delta\left( f_b\chi_{\mathcal F} \right),f_b\chi_{\mathcal F} \right\rangle=\int_M\left(f_b^{2}\left\vert \kappa \right\vert ^{2}+f_b^{2}\left\vert \varphi
_{0}\right\vert ^{2}+\left\vert df_b\right\vert ^{2}\right) dv_g.
\end{equation} 
Now substituting Equations \eqref{deltafbchifa} and \eqref{deltafbchifb} into Equation \eqref{splittingalpha},  we find that 
$$\left\langle \Delta _{a}\alpha ,\alpha \right\rangle=\left\langle \Delta\alpha _{1},\alpha _{1}\right\rangle
+\int_M\left(f_b^{2}\left\vert \kappa \right\vert ^{2}+\left\vert df_b\right\vert ^{2}\right)dv_g,$$
which is non-negative. Then $\left\langle \Delta _{a}\alpha ,\alpha \right\rangle=0$ if and only if $\alpha_1$ is harmonic (i.e. $\alpha_1\in H^1(M)$), 
$f_b$ is constant and $f_b\kappa=0.$ Recall that $\alpha=\alpha_1+\alpha_2$ with $\alpha_2=f_b\chi_{\mathcal F}.$ In the case where $H^1(M)=\{0\}$ and 
$\alpha$ is $\Delta_a$-harmonic $1$-form, then $\alpha_1=0$ and 
$\alpha=f_b\chi_{\mathcal F}={\rm (constant)}\,\chi_{\mathcal F}$.  But this constant cannot be zero in view of Remark \ref{rem:h1a}. Hence 
$\dim H_{a}^{1}\left( M,\mathcal{F}\right) =1.$ This proves the first part of the theorem. To prove the second part, we use the exact 
Gysin sequence for non-taut Riemannian flows established in \cite{RoP} 
\begin{equation*}
0\rightarrow H_{b}^{1}\left( M\right) \rightarrow H^{1}\left( M\right)
\rightarrow H_{\kappa ,b}^{0}\left( M\right) \rightarrow ...
\end{equation*}%
where $H_{\kappa ,b}^{0}\left( M\right) \cong H_{b}^{q}\left( M\right) $,
which is zero because the foliation is not taut. Thus, $H_{b}^{1}\left(
M\right) \cong H^{1}\left( M\right) $. By the proof of Proposition \ref%
{degree 1 cohomology Prop}, we get that $P_a\left(\mathcal{H}^1\left(M\right)\right)=0$. That means for every harmonic one-form $\omega$, we have $%
P_{a}\omega =0$, and thus is basic. Hence $\left\langle \Delta _{a}\alpha ,\alpha \right\rangle=0$ implies that $\alpha_1$ is basic-harmonic and $f_b=0.$ Thus both $\alpha_1$ and $\alpha_2$ are zero and then  $H_{a}^{1}\left( M,\mathcal{F}\right) =\{0\}.$  
\end{proof}

\begin{remark}
It might seem at first glance that Proposition \ref{H1zeroKbasicProp} may
contradict Proposition \ref{MeanBasicInvNormalBdleProp}. But in fact, if $%
H^{1}\left( M\right) =\left\{ 0\right\} $ for some compact manifold $M$,
then any Riemannian flow of $M$ must have a normal bundle that is not
involutive. The reason is as follows. First, the mean curvature can be
chosen to be zero after a change in bundle-like metric, since the mean
curvature must be exact. If the normal bundle is involutive, then $d\chi_{\mathcal F} =0$
from Rummler's formula, and $\delta \chi_{\mathcal F} =0$ (true for any Riemannian flow),
so that $\chi_{\mathcal F} $ is a harmonic one-form and therefore represents a nontrivial
class in $H^{1}\left( M\right) $, a contradiction. So Proposition \ref%
{MeanBasicInvNormalBdleProp} does not apply.
\end{remark}

\section{Examples\label{ExamplesSection}}

We illustrate the antibasic cohomology and our theorems in some
one-dimensional examples of foliations. 
These examples are certainly not meant to comprise a comprehensive list.

To simplify the exposition, we
denote the Betti numbers for each example foliation $\left( M,\mathcal{F}%
\right) $ as follows:%
\begin{equation*}
h^{j}=\dim H^{j}\left( M\right) ,~h_{b}^{j}=\dim H_{b}^{j}\left( M,\mathcal{F%
}\right) ,h_{a}^{j}=\dim H_{a}^{j}\left( M,\mathcal{F}\right) .
\end{equation*}

We start with the Hopf fibration, which is a taut Riemannian flow.

\begin{example}
\label{HopfFibExample}Using Theorem \ref{H1zeroKbasicProp} above, we
consider the Hopf fibration of $S^{3}\subseteq \mathbb{C}^{2}\rightarrow 
\mathbb{C}P^{1}$ via $\left( z_{0},z_{1}\right) \rightarrow \left[
z_{0},z_{1}\right] $. The leaves of the foliation $\mathcal{F}$ are the
orbits of the $S^{1}$ action $e^{it}\mapsto \left(
e^{it}z_{0},e^{it}z_{1}\right) $. This is a Riemannian flow, but the normal
bundle is not involutive. The lengths of the circular leaves are constant,
so the mean curvature is zero. By Theorem \ref{H1zeroKbasicProp}, $%
h_{a}^{1}=1$, and from Proposition \ref{groupActionProp}, $h_{a}^{0}\cong 0$%
. Also $H_{a}^{2}\left( S^{3},\mathcal{F}\right) \subseteq H^{2}\left(
S^{3}\right) $ because of Lemma \ref{qLemma}, so that $h_{a}^{2}=0$, and $%
h_{a}^{3}=h^{3}=1.$ In summary, we have%
\begin{eqnarray*}
(h^{0},h^{1},h^{2},h^{3}) &=&(1,0,0,1) ,\\
(h_{b}^{0},h_{b}^{1},h_{b}^{2}) &=&(1,0,1), \\
(h_{a}^{0},h_{a}^{1},h_{a}^{2},h_{a}^{3}) &=&(0,1,0,1).
\end{eqnarray*}%
\end{example}

The following example is a Riemannian flow of a 3-manifold that is not taut.

\begin{example}
\label{CarExample}We will compute the antibasic cohomology groups of the
Carri\`{e}re example from \cite{Car} in the $3$-dimensional case. Let $%
A=\left( 
\begin{array}{cc}
2 & 1 \\ 
1 & 1%
\end{array}%
\right) $. We denote respectively by $V_{1}$ and $V_{2}$ the eigenvectors
associated with the eigenvalues $\lambda $ and $\frac{1}{\lambda }$ of $A$
with $\lambda >1$ irrational. Let the hyperbolic torus $\mathbb{T}_{A}^{3}$
be the quotient of $\mathbb{T}^{2}\times \mathbb{R}$ by the equivalence
relation which identifies $(m,t)$ to $(A(m),t+1)$. The flow generated by the
vector field $V_{2}$ is a transversally Lie foliation of the affine group.
The Betti numbers of this closed manifold are $h^{j}=1$, for $0\leq j\leq 3$. 
We choose the bundle-like metric (letting $\left( x,s,t\right) $ denote the
local coordinates in the $V_{2}$ direction, $V_{1}$ direction, and $\mathbb{R%
}$ direction, respectively) as 
\begin{equation*}
g=\lambda ^{-2t}dx^{2}+\lambda ^{2t}ds^{2}+dt^{2}.
\end{equation*}%
The mean curvature of the flow is $\kappa =\kappa _{b}=\log \left( \lambda
\right) dt$, since $\chi _{\mathcal{F}}=\lambda ^{-t}dx$ is the
characteristic form and $d\chi _{\mathcal{F}}=-\log \left( \lambda \right)
\lambda ^{-t}dt\wedge dx=-\kappa \wedge \chi _{\mathcal{F}}$. It is easily
seen that the basic cohomology satisfies $h_{b}^{j}=1$ for $j=0,1$ and $%
h_{b}^{2}=0$ (class of the mean curvature class being nonzero implies this;
see \cite{AL}).
The foliation has an involutive normal bundle, so that
Proposition \ref{MeanBasicInvNormalBdleProp} applies, so that $h_{a}^{2}=1$, $h_{a}^{3}=1$ and $h_{a}^{k}=0$ for $k=0,1$. In summary,%
\begin{eqnarray*}
(h^{0},h^{1},h^{2},h^{3}) &=&(1,1,1,1) ,\\
(h_{b}^{0},h_{b}^{1},h_{b}^{2}) &=&(1,1,0) ,\\
(h_{a}^{0},h_{a}^{1},h_{a}^{2},h_{a}^{3}) &=&(0,0,1,1).
\end{eqnarray*}
\end{example}

We now consider an example of a foliation that is not Riemannian (for any
metric). This is a standard example of a foliation on a connected, compact
manifold with infinite-dimensional basic cohomology; this example is from \cite{Ghys}.

\begin{example}
\label{nonRiemFoliationFirstExample}Let $M$ be the closed 3-manifold defined
as $\mathbb{R}\times T^{2}\diagup \mathbb{Z}$, where $T^{2}=\mathbb{R}%
^{2}\diagup \mathbb{Z}^{2}$ and $m\in \mathbb{Z}$ acts on $\mathbb{R}\times
T^{2}$ by $m\left( t,x\right) =\left( t+m,A^{m}x\right) $, where $A$ is the
matrix $\left( 
\begin{array}{cc}
1 & 1 \\ 
0 & 1%
\end{array}%
\right) $. We define the leaves of the foliation to be the $t$-parameter
curves. Then observe that that leaf closures intersect each torus with a set
of the form $S\times \left\{ x_{2}\right\} $, where $x_{2}\in \mathbb{R}%
\diagup \mathbb{Z}$ and $S$ is a finite number of points for rational $x_{2}$
and is $\mathbb{R}\diagup \mathbb{Z}$ for irrational $x_{2}$. Thus, the
basic forms in the \textquotedblleft coordinates\textquotedblright\ $\left(
t,x_{1},x_{2}\right) $ have the form%
\begin{eqnarray*}
\Omega _{b}^{0}\left( M\right) &=&\left\{ f\in\Omega^0(M):f(t,x_1,x_2) \text{ is constant in }x_1,t\right\},
\\
\Omega _{b}^{1}\left( M\right) &=&\left\{ f\,dx_{2}:f\in\Omega _b^0(M)
\right\}, \\
\Omega _{b}^{2}\left( M\right) &=&\left\{ f\,dx_1\wedge dx_{2}:f\in\Omega _b^0(M)
\right\}. 
\end{eqnarray*}
From this we can easily calculate that $h_{b}^{0}=1$, $h_{b}^{1}=1$, and $%
H_{b}^{2}\left( M,\mathcal{F}\right) \cong \Omega _{b}^{2}\left( M\right) ,$
which is infinite dimensional. One may also check with a cell complex that
the ordinary homology satisfies $H^{j}\left( M,\mathbb{Z}\right) \cong 
\mathbb{Z}$ for $j=0,3$ and $H^{j}\left( M,\mathbb{Z}\right) \cong \mathbb{Z}%
^{2}$ for $j=1,2$, so that the ordinary de Rham cohomology satisfies $%
h^{j}=1 $ for $j=0,3$ and $h^{j}=2$ for $j=1,2$. We choose the metric 
in the 
$\partial_t,\partial_{x_1},\partial_{x_2}$ basis as 
\begin{equation*}
\left( g_{ij}\right) =\left( 
\begin{array}{ccc}
1 &0 & 0\\ 
0 &1 & -t \\ 
0 &-t & 1+t^{2}
\end{array}
\right) .
\end{equation*}
One can check the invariance with respect to the action of 
$m\in \mathbb{Z}$; it is chosen so that\linebreak
$
\left\{ e_1=\partial_t, e_2=\partial _{x_{1}}, e_3=t\partial
_{x_{1}}+\partial _{x_{2}}\right\} $ 
forms an orthonormal basis at each $(t,x_1,x_2)$.
Then the metric on covectors with basis $\{ dt,dx_1,dx_2\}$ is
\begin{equation*}
\left( g^{ij}\right) =\left( 
\begin{array}{ccc}
1 & 0 & 0 \\
0 & 1+t^{2} & t \\ 
0 & t & 1
\end{array}
\right) ,
\end{equation*}
and the dual orthonormal basis is $\left\{ e^1=dt,e^2=dx_{1}-tdx_{2},e^3=dx_{2}\right\} $. Note also that 
\linebreak $g=\det (g_{ij})=1$.

We now compute the antibasic forms with respect to this metric, which are the sets of smooth forms that are $L^2$-orthogonal to the 
sets of basic forms listed above. 
For any $x_2\in \mathbb{R}\diagup \mathbb{Z}$, let $C_{x_2}$ denote the torus $\left\{ \left(
t,x_1,x_2 \right) :t,x_{1}\in \mathbb{R}\diagup \mathbb{Z}\right\} $. Then:
\begin{eqnarray*}
\Omega _a^0 ( M ) &=&\left\{ f\in\Omega^0(M):\int_{C_{x_2}}f\left( t, x_{1},x_{2}\right) ~dt\wedge dx_{1}=0
\text{ for all }x_2\in \mathbb{R}\diagup \mathbb{Z}\right\}, \\
\Omega _{a}^{1}\left( M\right) &=&\left\{  f\,
dx_{2}+
g\,\left( dx_{1}-tdx_{2}\right) 
+
h\,dt: f\in\Omega_a^0(M), g,h\in\Omega^0(M) \right\}, \\
\Omega _{a}^{2}\left( M\right) &=&\left\{ f\,dx_{1}\wedge dx_{2}+g\, dt\wedge \left(
dx_{1}-tdx_{2}\right) +h\, dt\wedge dx_{2}: f\in\Omega_a^0(M), g,h\in\Omega^0(M) \right\}, \\
\Omega _{a}^{3}\left( M\right) &=&\Omega ^{3}\left( M\right) .
\end{eqnarray*}

Immediately we have $h_{a}^{3}=1$. We compute the divergence on one-forms:
\begin{eqnarray*}
\left\langle df,adt+c_{1}\left( dx_{1}-tdx_{2}\right)
+c_{2}dx_{2}\right\rangle &=&\int f_{t}a+f_{x_{1}}c_{1}+\left(
tf_{x_{1}}+f_{x_{2}}\right) c_{2} \\ 
&=&\int f\left( -a_{t}-\left( c_{1}\right) _{x_{1}}-t\left( c_{2}\right)
_{x_{1}}-\left( c_{2}\right) _{x_{2}}\right); \\ 
\delta \left( adt+c_{1}\left( dx_{1}-tdx_{2}\right) +c_{2}dx_{2}\right)
&=&-a_{t}-\partial _{1}\left( c_{1}\right) -\left( t\partial _{1}+\partial
_{2}\right) \left( c_{2}\right) ,
\end{eqnarray*}
which makes sense since all three vector fields are divergence-free.

We next compute the divergence of $2$-forms, writing in terms of our frame and coframe. 
Note that in
terms of an orthonormal 
frame $\left\{ e_{i}\right\} $ 
with Christoffel
symbols defined by $\nabla _{e_{i}}e_{j}=G _{ij}^{k}e_{k}$ or $d\left(
e^{k}\right) =-G _{ij}^{k}
e^{i}\wedge e^{j}$ 
(using Einstein summation convention here and subsequently)
we have 
$\delta\left( b_{ij}e^{i}\wedge e^{j}\right) =
-e_i(b_{ij})e^j+e_j(b_{ij})e^i
+b_{ij}
\left(
G_{\ell\ell}^i
\delta^j_r
-G_{\ell\ell}^j\delta^i_r
-G_{jr}^i+G_{ir}^j\right)
e^r$.
In what follows, we will assume that the two form is anti-symmetrized, so that
$b_{ji}=-b_{ij}$, and then the formula above simplifies to
$\delta(b_{ij}e^{i}\wedge e^{j})=-2e_{i}(b_{ij})e^{j}+2b_{ij}(G_{\ell\ell}^{i}\delta_{r}^{j}-G_{jr}^{i})e^{r}$.
In our case with
$e_1=\partial_t$, $e_2=\partial_{x_1}$, $e_3=t\partial_{x_1}+\partial_{x_2}$, the covariant derivatives give
$\frac 12=G_{13}^2=G_{23}^1=G_{32}^1=-G_{12}^3=-G_{21}^3=-G_{31}^2$
with all other  $G_{ij}^k$ zero.
Also, all Lie brackets between these basis vector fields are zero except that
$
[e_1,e_3]=-[e_1,e_3] =e_2
$.
After a bit of calculation, we have (for antisymmetrized  $b_{ij}e^{i}\wedge e^{j}$)
\begin{multline*}
\delta \left( b_{ij}e^{i}\wedge
e^{j}\right) \\
=\left( -e_{2}\left( 2b_{21}\right) -e_{3}\left( 2b_{31}\right) \right)
e^{1}
+\left( -e_{1}\left( 2b_{12}\right) -e_{3}\left( 2b_{32}\right)
-2b_{13}\right) e^{2}
+\left( -e_{1}\left( 2b_{13}\right) -e_{2}\left( 2b_{23}\right) \right)
e^{3},
\end{multline*}
which implies also that (substituting $b_{ij}e^i\wedge e^j$ below with $\frac 12 b_{ij}e^i\wedge e^j - \frac 12 b_{ij}e^j\wedge e^i$ above)
\begin{eqnarray}
&\,&\delta \left( \sum_{i<j}b_{ij}e^{i}\wedge
e^{j}\right) \notag \\
&=&\left( e_{2}\left( b_{12}\right) +e_{3}\left( b_{13}\right) \right)
e^{1}
+\left( -e_{1}\left( b_{12}\right) +e_{3}\left( b_{23}\right)
-b_{13}\right) e^{2}
+\left( -e_{1}\left( b_{13}\right) -e_{2}\left( b_{23}\right) \right)
e^{3}. \label{div_oneforms}
\end{eqnarray}

Finally we calculate divergence of $3$-forms:
\begin{eqnarray*}
\delta \left( f~e^{1}\wedge e^{2}\wedge e^{3}\right)&=&-\ast d\ast \left(
f~e^{1}\wedge e^{2}\wedge e^{3}\right) =-\ast \left( df\right) \\
&=&-\frac{1}{2}\sum_{\sigma \in S_{3}}\limfunc{sgn}\left( \sigma \right)
e_{\sigma _{1}}\left( f\right) ~e^{\sigma _{2}}\wedge e^{\sigma _{3}}\\
&=& -e_1(f) e^2\wedge e^3-e_2(f) e^3\wedge e^1 - e_3(f) e^1\wedge e^2.
\end{eqnarray*}

From these formulas, we note that the divergence of a basic one-form (one of
the type $c_{2}\left( x_{2}\right) dx_{2}$) is always a basic function ($%
-\partial _{2}c_{2}$), so that $\delta $ maps basic one-forms to basic
functions and antibasic one-forms to antibasic functions. Then $h_{a}^{0}=0$, 
since in this case 
\begin{eqnarray*}
H^{0}\left( M\right) &=&\frac{\Omega ^{0}}{\func{im}\left. \delta
\right\vert _{\Omega ^{1}}}=\frac{\Omega _{a}^{0}\oplus \Omega _{b}^{0}}{%
\left( \func{im}\left. \delta \right\vert _{\Omega _{a}^{1}}\right) \oplus
\left( \func{im}\left. \delta \right\vert _{\Omega _{b}^{1}}\right) } \\
&=&\frac{\Omega _{a}^{0}}{\left( \func{im}\left. \delta \right\vert _{\Omega
_{a}^{1}}\right) }\oplus \frac{\Omega _{b}^{0}}{\left( \func{im}\left.
\delta \right\vert _{\Omega _{b}^{1}}\right) } \\
&=&H_{a}^{0}\left( M,\mathcal{F}\right) \oplus H_{b}^{0}\left( M,\mathcal{F}%
\right) \cong H_{a}^{0}\left( M,\mathcal{F}\right) \oplus H^{0}\left( M\right) .
\end{eqnarray*}%
(Note that the first and second step fail for foliations in general).

We now compute $H_{a}^{2}\left( M,\mathcal{F}\right) $. We have%
\begin{equation*}
H_{a}^{2}\left( M,\mathcal{F}\right) =\frac{\ker \left. \delta \right\vert
_{\Omega _{a}^{2}}}{\func{im}\left. \delta \right\vert _{\Omega ^{3}}}%
\subseteq \frac{\ker \left. \delta \right\vert _{\Omega ^{2}}}{\func{im}%
\left. \delta \right\vert _{\Omega ^{3}}}=H^{2}\left( M\right) \cong \mathbb{%
R}^{2}.
\end{equation*}%
In the ordinary $\delta $-cohomology, the generators of $H^{2}\left(
M\right) $ are $\left[ e^1\wedge e^2=dt\wedge \left( dx_{1}-tdx_{2}\right) \right] $ and $%
\left[ e^2\wedge e^3 = dx_{1}\wedge dx_{2}\right] $. But $b_{12}dt\wedge \left(
dx_{1}-tdx_{2}\right)=b_{12}e^1\wedge e^2$ 
is antibasic and $b_{23}dx_{1}\wedge dx_{2}$ is
basic for any choice of constants $b_{12}$, $b_{23}$. Thus, $h_{a}^{2}=1$.

We now consider $1$-forms. Observe that basic $2$-forms have the form 
$b_{23}\left( x_{2}\right) dx_{1}\wedge dx_{2}=b_{23}\left( x_{2}\right) e^2\wedge e^3$,
and from formula (\ref{div_oneforms}) above 
$%
\delta \left( b_{23}\left( x_{2}\right) dx_{1}\wedge dx_{2}\right) 
=b_{23}'(x_2)e^2=b_{23}'(x_2)(dx^1-tdx^2)$ for
all basic functions $b_{23}$; note that the image is an antibasic one-form.
It follows that the image of $\delta _{a}$ is a proper subset of the image of $%
\delta $ 
on $2$-forms, which is contained in the space of $\delta$-closed antibasic one-forms. 
Thus, we have%
\begin{equation*}
H_{a}^{1}\left( M,\mathcal{F}\right) =\frac{\ker \left. \delta \right\vert
_{\Omega _{a}^{1}}}{\func{im}\left. \delta \right\vert _{\Omega _{a}^{2}}}
\twoheadrightarrow%
\frac{\ker \left. \delta \right\vert _{\Omega _{a}^{1}}}{\func{im}\left.
\delta \right\vert _{\Omega ^{2}}}\subseteq \frac{\ker \left. \delta
\right\vert _{\Omega ^{1}}}{\func{im}\left. \delta \right\vert _{\Omega ^{2}}%
}=H^{1}\left( M\right) \cong \mathbb{R}^{2}.
\end{equation*}%
In fact, we will show  
$H_{a}^{1}( M,\mathcal{F}) $  is infinite-dimensional. Consider the subspace
$\{ F'(x_2)e^2:F\in C^\infty(\mathbb{R}\slash\mathbb{Z})\}\slash \ker\left.\delta\right|_{\Omega_a^2}$ of $H_{a}^{1}( M,\mathcal{F}) $;
The set $\{ F'(x_2)e^2\}$ is clearly infinite-dimensional and is a subspace of 
$\ker\left.
\delta\right|_{\Omega_a^1}$. 
We wish to determine when two elements of this space are equivalent mod $\delta(\Omega_a^2)$.
If $\delta(\beta)=F'(x_2) e^2$, then by Hodge theory
$\beta = -F(x_2)e^2\wedge e^3+($harmonic $2$-form$)+($element of 
$\func{Im}\left.\delta\right|_{\Omega^3} )$. The first term is basic, and the second term is 
(constant)$e^2\wedge e^3+$ (constant)$e^1\wedge e^2$, so we may rewrite
$\beta =  -(F(x_2)+c_1)e^2\wedge e^3+$ (antibasic 2-form). The form $\beta$ can be 
antibasic if and only if $F(x_2)+c_1=0$, so we conclude that the space
$\{ F'(x_2)e^2\}\slash \delta(\Omega_a^2)$ is infinite-dimensional. Thus, $h_a^1=\infty$.

In summary,%
\begin{eqnarray*}
(h^{0},h^{1},h^{2},h^{3}) &=&(1,2,2,1), \\
(h_{b}^{0},h_{b}^{1},h_{b}^{2}) &=&(1,1,\infty ) ,\\
(h_{a}^{0},h_{a}^{1},h_{a}^{2},h_{a}^{3}) &=&(0,\infty,1,1).
\end{eqnarray*}
\end{example}

The following non-Riemannian flow is a simple example where the basic
projection $P_{b}$ and antibasic projection $P_{a}$ do not preserve
smoothness. In spite of that, the basic and antibasic cohomology can be
calculated. From the calculations, we also see that the Hodge theorem is
false for this foliation.

\begin{example}
\label{ProjectionNotSmoothExample}Let $M$ be the flat torus $[0,2]\times
\lbrack 0,1]$ with opposite sides of the boundary identified. Let $\phi
\left( x\right) $ be a smooth function on the circle $[0,2]\func{mod}2$ such
that $\phi $ is positive and $\leq 1$ on $(0,1)$ and identically zero on $[1,2]$.
Consider the foliation that whose tangent space at each point is spanned by
the vector field $V(x,y)=\left( \phi \left( x\right) ,\sqrt{1-\phi \left(
x\right) ^{2}}\right) $. All the leaves in the region $0<x<1$ are noncompact
and have $x=0$ and $x=1$ in their closure, and the leaves in the region $%
1\leq x\leq 2$ are vertical circles. The set of basic functions is%
\begin{equation*}
\Omega _{b}^{0}(M)=\left\{ 
\begin{array}{c}
f:[0,2]\times \lbrack 0,1]\rightarrow \mathbb{R}:f\left( x,y\right) =g\left(
x\right) \text{ for a smooth function }g\text{ } \\ 
\text{on }\mathbb{R}\diagup 2\mathbb{Z}\text{ such that }g\left( x\right) =\text{%
constant for }x\in \lbrack 0,1]\func{mod}2%
\end{array}%
\right\} .
\end{equation*}%
Since every basic normal vector field approaches $0$ as $x\rightarrow 0^{+}$
and $x\rightarrow 1^{-}$, there are no bounded basic one-forms for $0<x<1$,
so we have%
\begin{equation*}
\Omega _{b}^{1}(M)=\left\{ 
\begin{array}{c}
\omega =h\left( x\right) dx:h\text{ is a smooth function on }\mathbb{%
R}\diagup 2\mathbb{Z}\text{ } \\ 
\text{such that }h\left( x\right) =0\text{ for }x\in \lbrack 0,1]\func{mod}2%
\end{array}%
\right\} .
\end{equation*}%
Then, the set of antibasic forms are those smooth forms orthogonal to $%
\Omega _{b}^{\ast }\left( M\right) $ in $L^{2}$. We obtain%
\begin{eqnarray*}
\Omega _{a}^{0}(M) &=&\left\{ f:M\rightarrow \mathbb{R}%
:\int_{0}^{1}f\left( x,y\right) dy=0\text{ for }x\in \lbrack 1,2]\func{mod}2%
\text{ and }\int_{0}^{1}\int_{0}^{1}f\left( x,y\right) dx~dy=0\text{ }%
\right\} , \\
\Omega _{a}^{1}\left( M\right) &=&\left\{ \alpha
=a_{1}dx+a_{2}dy:\int_{0}^{1}a_{1}\left( x,y\right) dy=0\text{ for }x\in
\lbrack 1,2]\func{mod}2.\right\} .
\end{eqnarray*}%
Note that in this example, the basic and antibasic projections are not
smooth maps to differential forms. Observe that on functions, 
\begin{equation*}
P_{b}\left( f\right) \left( x,y\right) =\left\{ 
\begin{array}{ll}
\text{average of }f\text{ over }[0,1]^{2} & x\in (0,1) \\ 
\text{average of }f\left( x,\cdot \right) \text{ over }[0,1] & x\in (1,2)%
\end{array}%
\right. ,
\end{equation*}%
so for instance 
\begin{equation*}
P_{b}(\sin (\pi x))=\left\{ 
\begin{array}{cc}
\frac{2}{\pi } & x\in (0,1) \\ 
\sin (\pi x) & x\in (1,2)%
\end{array}%
\right. \newline
,
\end{equation*}%
which is not a continuous function. Likewise, $P_{a}\left( \sin \left( \pi
x\right) \right) =\sin (\pi x)-P_{b}(\sin (\pi x))$ is not smooth. \newline
In any case, we may calculate the basic and antibasic cohomology groups. 
\begin{eqnarray*}
H_{b}^{0}\left( M,\mathcal{F}\right) &=&\ker \left( d:\Omega
_{b}^{0}(M)\rightarrow \Omega _{b}^{1}(M)\right) \cong \mathbb{R,} \\
H_{b}^{1}\left( M,\mathcal{F}\right) &=&\frac{\Omega _{b}^{1}(M)}{\func{im}%
\left( d:\Omega _{b}^{0}(M)\rightarrow \Omega _{b}^{1}(M)\right) } \\
&=&\frac{\Omega _{b}^{1}(M)}{\left\{ h\left( x\right) dx:\int_{1}^{2}h\left(
x\right) dx=0\text{ and }h\left( x\right) =0\text{ for }x\in \lbrack
0,1]\right\} } \\
&=&\left\{ [c\cdot \text{bump on }[0,1]]\right\} \cong \mathbb{R.}
\end{eqnarray*}%
Note that $H_{b}^{1}$ is not represented by a basic harmonic form. Also,%
\begin{eqnarray*}
H_{a}^{0}\left( M,\mathcal{F}\right) &=&\frac{\Omega _{a}^{0}(M)}{\func{im}%
\left( \delta :\Omega _{a}^{1}(M)\rightarrow \Omega _{a}^{0}(M)\right) } \\
&=&\frac{\Omega _{a}^{0}(M)}{\left\{ \left( a_{1}\right) _{x}+\left(
a_{2}\right) _{y}:\int_{0}^{1}a_{1}\left( x,y\right) dy=0\text{ for }x\in
\lbrack 1,2]\func{mod}2\right\} }=\left\{ 0\right\}, \\
H_{a}^{1}\left( M,\mathcal{F}\right) &=&\frac{\ker \left( \delta :\Omega
_{a}^{1}(M)\rightarrow \Omega _{a}^{0}(M)\right) }{\func{im}\left( \delta
:\Omega ^{2}(M)\rightarrow \Omega _{a}^{1}(M)\right) } \\
&=&\frac{\left\{ a_{1}dx+a_{2}dy:\int_{0}^{1}a_{1}\left( x,y\right) dy=0%
\text{ for }x\in \lbrack 1,2]\func{mod}2\text{ and }\left( a_{1}\right)
_{x}+\left( a_{2}\right) _{y}=0\right\} }{\left\{ f_{y}dx-f_{x}dy\right\} }
\\
&=&\left\{ \left[ c~dy\right] \right\} \cong \mathbb{R},\\
H_{a}^{2}\left( M,\mathcal{F}\right) &=&H^{2}\left( M\right) \cong \mathbb{R.%
}
\end{eqnarray*}%
And we also have $h^{0}=h^{2}=1$, $h^{1}=2$. In summary,%
\begin{eqnarray*}
(h^{0},h^{1},h^{2}) &=&(1,2,1), \\
(h_{b}^{0},h_{b}^{1}) &=&(1,1), \\
(h_{a}^{0},h_{a}^{1},h_{a}^{2}) &=&(0,1,1).
\end{eqnarray*}
\end{example}

\end{document}